\documentclass{amsart}
\usepackage{amssymb,amsmath,amsthm,amsfonts}
\usepackage {latexsym,enumerate}
\usepackage{bbm}

\allowdisplaybreaks
\newcommand{\nc}{\newcommand}
\nc{\les}{\lesssim}
\nc{\nit}{\noindent}
\nc{\nn}{\nonumber}
\nc{\D}{\partial}
\nc{\diff}[2]{\frac{d #1}{d #2}}
\nc{\diffn}[3]{\frac{d^{#3} #1}{d {#2}^{#3}}}
\nc{\pdiff}[2]{\frac{\partial #1}{\partial #2}}
\nc{\pdiffn}[3]{\frac{\partial^{#3} #1}{\partial{#2}^{#3}}}
\nc{\abs}[1] {\lvert #1 \rvert}
\nc{\cAc}{{\cal A}_c}
\nc{\cE}{{\cal E}}
\nc{\cF}{{\cal F}}
\nc{\cP}{{\cal P}}
\nc{\cV}{{\cal V}}
\nc{\cQ}{{\cal Q}}
\nc{\cGin}{{\cal G}_{\rm in}}
\nc{\cGout}{{\cal G}_{\rm out}}
\nc{\cO}{{\cal O}}
\nc{\Lav}{{\cal L}_{\rm av}}
\nc{\cL}{{\cal L}}
\nc{\cB}{{\cal B}}
\nc{\cZ}{{\cal Z}}
\nc{\cR}{{\cal R}}
\nc{\cT}{{\cal T}}
\nc{\cY}{{\cal Y}}
\nc{\cX}{{\cal X}}
\nc{\cXT}{{{\cal X}(T)}}
\nc{\cBT}{{{\cal B}(T)}}
\nc{\vD}{{\vec \mathcal{D}}}
\nc{\efield}{\mathcal{E}}
\nc{\vE}{{\vec \efield}}
\nc{\vB}{{\vec \mathcal{B}}}
\nc{\vH}{{\vec \mathcal{H}}}
\nc{\mR}{\mathcal R}
\nc{\mF}{\mathcal F}
\nc{\ty}{{\tilde y}}
\nc{\tu}{{\tilde u}}
\nc{\tV}{{\tilde V}}
\nc{\Pc}{{\bf P_c}}
\nc{\bx}{{\bf x}}
\nc{\bX}{{\bf X}}
\nc{\bXYZ}{{\bf XYZ}}
\nc{\bY}{{\bf Y}}
\nc{\bF}{{\bf F}}
\nc{\bS}{{\bf S}}
\nc{\dV}{{\delta V}}
\nc{\dE}{{\delta E}}
\nc{\TT}{{\Theta}}
\nc{\dPsi}{{\delta\Psi}}
\nc{\order}{{\cal O}}
\nc{\Rout}{R_{\rm out}}
\nc{\eplus}{e_+}
\nc{\eminus}{e_-}
\nc{\epm}{e_\pm}
\nc{\sgn}{\text{sgn}}
\nc{\eps}{\varepsilon}
\nc{\vnabla}{{\vec\nabla}}
\nc{\G}{\Gamma}
\nc{\w}{\omega}
\nc{\mh}{h}
\nc{\mg}{g}
\nc{\vphi}{\varphi}
\nc{\tlambda}{\tilde\lambda}
\nc{\be}{\begin{equation}}
\nc{\ee}{\end{equation}}
\nc{\ba}{\begin{eqnarray}}
\nc{\ea}{\end{eqnarray}}

\nc{\g}{\gamma}
\nc{\ol}{\overline}

\newtheorem{theorem}{Theorem}[section]
\newtheorem{lemma}[theorem]{Lemma}
\newtheorem{prop}[theorem]{Proposition}
\newtheorem{corollary}[theorem]{Corollary}

\nc{\pT}{\partial_T}
\nc{\pz}{\partial_z}
\nc{\pt}{\partial_t}
\nc{\la}{\langle}
\nc{\ra}{\rangle}
\nc{\infint}{\int_{-\infty}^{\infty}}
\nc{\halfwidth}{6.5cm}
\nc{\figwidth}{10cm}
\newcommand{\f}{\frac}

\nc{\nlayers}{L} \nc{\nsectors}{M}
\nc{\indicator}{\mathbf{1}}
\nc{\Rhole}{R_{\rm hole}}
\nc{\Rring}{R_{\rm ring}}
\nc{\neff}{n_{\rm eff}}
\nc{\Frem}{F_{\rm rem}}
\nc{\R}{\mathbb R}
\nc{\mJ}{\mathcal J}

\nc{\mE}{\mathcal E}
\nc{\C}{\mathbb C}
\nc{\Z}{\mathbb Z}
\nc{\N}{\mathbb N}
\nc{\DD}{\Delta}
\nc{\cD}{\mathcal D}
\nc{\lnorm}{\left\|}
\nc{\rnorm}{\right\|}
\nc{\rnormp}{\right\|_{\ell^{p,\eps}}}
\nc{\rar}{\rightarrow} 
\begin{document}
	
	\begin{abstract}
		
		We consider fractional Schr\"odinger operators $H=(-\Delta)^\alpha+V(x)$ in $n$ dimensions with real-valued potential $V$ when $n>2\alpha$, $\alpha>1$.  We show that the wave operators  
		extend to bounded operators on $L^p(\mathbb R^n)$ for all $1\leq p\leq\infty$  under conditions on the potential that depend on $n$ and $\alpha$ analogously to the case when $\alpha\in \mathbb N$.  As a consequence, we deduce a family of dispersive and Strichartz estimates for the perturbed fractional Schr\"odinger operator.
		
	\end{abstract}

	\title[Wave operators for fractional order  Schr\"odinger operators]{\textit{The $L^p$-continuity of wave operators for fractional order Schr\"odinger operators } } 
	
	\author[Erdo\smash{\u{g}}an, Goldberg, Green]{M. Burak Erdo\smash{\u{g}}an, Michael Goldberg and William~R. Green}
	\thanks{  The first author was partially supported by the NSF grant  DMS-2154031.  The second author is partially supported by Simons Foundation
		Grant 635369. The third author is partially supported by Simons Foundation
		Grant 511825. }
	\address{Department of Mathematics \\
		University of Illinois \\
		Urbana, IL 61801, U.S.A.}
	\email{berdogan@illinois.edu}
	\address{Department of Mathematics\\
		University of Cincinnati \\
		Cincinnati, OH 45221 U.S.A.}
	\email{goldbeml@ucmail.uc.edu}
	\address{Department of Mathematics\\
		Rose-Hulman Institute of Technology \\
		Terre Haute, IN 47803, U.S.A.}
	\email{green@rose-hulman.edu}

	\maketitle

	\section{Introduction}
	
	We study the non-local fractional  Schr\"odinger equation 
	\begin{align*}
	i\psi_t =(-\Delta)^\alpha  \psi +V\psi, \qquad x\in \R^n, \qquad \alpha>1, \quad  \alpha \not \in \mathbb N.
	\end{align*}
	We consider spatial dimensions $n>2\alpha $, and $V$ is a real-valued, decaying potential.  When $\alpha \not\in\mathbb N$, the non-local operator $(-\Delta)^\alpha$ is defined via the Fourier multiplier $|\xi|^{2\alpha}$, that is $(-\Delta)^\alpha f=\mathcal F^{-1}(|\xi|^{2\alpha}\widehat{f}(\xi))$.  We denote the free fractional Schr\"odinger operator by $H_0=(-\Delta)^\alpha  $, and the perturbed operator by $H=(-\Delta)^\alpha  +V(x)$. 
	
	Similar to the integer order Schr\"odinger operators, for the potentials we consider there is a Weyl criterion and $\sigma_{ac}(H)=\sigma_{ac}(H_0)=[0,\infty)$.  When $\alpha\neq 1$, decay of the potential is generally not sufficient to ensure the lack of eigenvalues embedded in the continuous spectrum for the fractional order operators.  In \cite{Cu embed}, Cuenin constructs examples where embedded eigenvalues  can occur for generalized Schr\"odinger operators. On the other hand, Ishida, L\H orinczi and Sasaki provided conditions on the potential when $0<\alpha<2$ in \cite{ILS} for which $H$ has no embedded eigenvalues.   We leave the lack of embedded eigenvalues as an overarching assumption.

	We study the $L^p$ boundedness of the wave operators, which are defined by
	$$
		W_{\pm}=s\text{\ --}\lim_{t\to \pm \infty} e^{itH}e^{-itH_0}.
	$$
	As in the classical   Schr\"odinger operators, see the work of Agmon, \cite{agmon}, H\"ormander, \cite{Hor} and  Schechter, \cite{Sche}, sufficient decay of the potential at infinity to ensures that the wave operators exist and are asymptotically complete, \cite{IW,ZHZ}.
	In particular we have the intertwining identity
	\begin{align}\label{eq:intertwining}
		f(H)P_{ac}(H)=W_\pm f((-\Delta)^\alpha  )W_{\pm}^*.
	\end{align}
	Here $P_{ac}(H)$ is the projection onto the absolutely continuous spectral subspace of $H$, and $f$ is any Borel function.
	 Asymptotic completeness and lack of embedded eigenvalues for magnetic fractional equations were studied in \cite{WD}.  For a more detailed discussion of the existence and asymptotic completeness see \cite{ZHZ}.

	Fractional Schr\"odinger equations have garnered interest in the physics literature, see for example \cite{Laskin,Laskin2}, where the Feynman path integral is taken over L\'evy flight paths in place of Brownian paths.  For applications in optics see \cite{Loptics}, further applications are considered in \cite{GX}.  Mathematically, one can view fractional Schr\"odinger equations as models of nonlocal dispersive equations.

	We use the notation $\la x\ra$ to denote $  (1+|x|^2)^{\f12}$, $\mathcal F(f)$ or $\widehat f$ to denote the Fourier transform of $f$.  We write $A\les B$ to say that there exists a constant $C$ with $A\leq C B$, and write $a-:=a-\epsilon$ and $a+:=a+\epsilon$ for some $\epsilon>0$ throughout the paper.  We use the norm $\|f\|_{H^{\delta}}=\|\la \cdot \ra^{\delta} \widehat f(\cdot)\|_2$.  We define zero energy to be regular if there are no non-trivial distributional solutions to $H\psi=0$ with $\la x \ra^{-\alpha -}\psi(x) \in L^2(\R^n)$ when $2\alpha<n\leq 4\alpha$ and $\psi\in L^2(\R^n)$ when $n>4\alpha$, which correspond to resonances or eigenvalues respectively, see \cite{EGG frac disp}.  We show   We first state a small potential result.
	\begin{theorem}\label{thm:small}
		Fix $\alpha>1$ and let $n>2\alpha $. 
		Assume that   $V$ is a real-valued potential on $\R^n$ and   fix $0<\delta\ll 1$. Then $\exists C=C(\delta,n,\alpha)>0 $ so that the wave operators  extend   to   bounded operators on $L^p(\R^n)$ for all $1\leq p\leq \infty$, provided that 
		\begin{enumerate}[i)]
			
			\item $\big\| \la \cdot \ra^{\frac{4\alpha +1-n}{2}+\delta} V(\cdot)\big\|_{2}<C$   when $2\alpha <n<4\alpha -1$,
			
			\item $\big\|\la \cdot \ra^{1+\delta}V(\cdot)\big\|_{H^{\delta}}<C$ when $n=4\alpha -1$,
			
			\item  $\big\|\mathcal F(\la \cdot \ra^{\sigma} V(\cdot))\big\|_{L^{ \frac{n-1-\delta}{{n-2\alpha}-\delta} }}<C$ for some $\sigma>\frac{2{n-4\alpha}}{n-1-\delta}+\delta$   when $n>4\alpha -1$,
			
			\item $H=(-\Delta)^\alpha  +V(x)$ has no positive
			eigenvalues and zero energy is regular. 
			
		\end{enumerate}
		
	\end{theorem}
	The assumptions on the potential are the generalizations of the $\alpha=m\in \mathbb N$ case studied in \cite{EGWaveOp} obtained by bounding the contribution of the Born series terms in Section~\ref{sec:Born} below. 
	There are technical hurdles to overcome to adapt the argument to the non-local fractional Schr\"odinger operators, and the analysis is rather delicate.

	Furthermore, one may remove the smallness assumption above provided $V$ decays sufficiently at spatial infinity.  We write $n_{\star}$ to denote $n+4$ if $n$ is odd and $n+3$ if $n$ is even.  
	\begin{theorem}\label{thm:main}
		Fix $\alpha>1$ and let $n>2\alpha $.
		Assume that  $V$ is a real-valued potential on $\R^n$ so that
		\begin{enumerate}[i)]
			\item $|V(x)|\les \la x\ra^{-\beta}$ for some $\beta>n_{\star}$ ,

			\item $\|\la \cdot \ra^{1+}V(\cdot)\|_{H^{0+}}<\infty$   when $n=4\alpha -1$,
			
			\item for some $0<\delta\ll 1$ and $\sigma>\frac{2{n-4\alpha}}{n-1-\delta}$, $\|\mathcal F(\la \cdot \ra^{\sigma} V(\cdot))\|_{L^{ \frac{n-1-\delta}{{n-2\alpha}-\delta} }}<\infty$   when $n>4\alpha -1$,
			
			\item $H=(-\Delta)^\alpha  +V(x)$ has no positive
			eigenvalues and zero energy is regular. 
		\end{enumerate}
	Then,	the wave operators extend to bounded operators on $L^p(\R^n)$ for all $1\leq p\leq \infty$.  
 \end{theorem}

	These results are, to the best of our knowledge, the first results studying $L^p$-boundedness of the wave operators for the non-local fractional Schr\"odinger operators.

	  From the intertwining identity \eqref{eq:intertwining} one may obtain $L^p$-based mapping properties for the more complicated, perturbed operator $f(H)P_{ac}(H)$ from the simpler free operator $f((-\Delta)^\alpha)$.  The boundedness of the wave operators on $L^p(\R^n)$ for any choice of $p\geq 2$ with the function $f(\cdot)=e^{-it(\cdot)}$ yield the family of dispersive estimates
	\begin{corollary}\label{cor:disp ests}
		
		Under the conditions of Theorem~\ref{thm:small} or \ref{thm:main}, for any $1\leq p\leq 2$ we have the following family of dispersive bounds
		\begin{align}\label{eqn:disp est}
			\|e^{-itH}P_{ac}(H)\|_{L^{p}\to L^{p'}}\les |t|^{\frac{n}{\alpha }(\frac12-\frac1{p}) },
		\end{align}
		where $p'$ is the H\"older conjugate of $p$.  
		
	\end{corollary}
	In particular in all dimensions $n>2\alpha $,  we have the global bounds
	$$
		\|e^{-itH}P_{ac}(H)\|_{L^{1}\to L^\infty}\les |t|^{-\frac{n}{2\alpha } },
	$$	 
	which extends the recent work of the authors, \cite{EGG frac disp}, to dimensions $n>4\alpha-1$.  
	Another consequence, following the seminal work of Ginibre and Velo, \cite{GVStrich}, is a family of Strichartz estimates:
	\begin{corollary}
		
		Under the conditions of Theorem~\ref{thm:small} or \ref{thm:main}, we have
		$$
			\|e^{-itH}P_{ac}(H)f\|_{L^q_t L^r_x}\les \|f\|_{L^2}, \quad \frac{2}{q}=\frac{n}{\alpha}\bigg(\frac{1}{2}-\frac{1}{r}\bigg), \quad 2\leq r<\infty.
		$$
		
	\end{corollary}
	
	Noting that the Fourier transform of $e^{i|\xi|^{2\alpha}}|\xi|^{\gamma - n}$ is bounded  when $\frac12 < \alpha $ and $0  < \gamma \leq n\alpha$, by scaling the free operator satisfies the bounds
	$$
		\| e^{-it(-\Delta)^{\alpha}} (-\Delta)^{\frac{\gamma-n}2} \|_{L^1\to L^{\infty}}\les |t|^{-\frac{\gamma}{2\alpha}}.
	$$
	We have the following corollary.
	Similar bounds were proved in \cite{EGG wiener} for integer $\alpha\in (\frac{n}{4},\frac{n}2)$.
	\begin{corollary}\label{cor:disp ests2}
		
		Under the conditions of Theorem~\ref{thm:small} or \ref{thm:main}, for any $0<\gamma\leq n\alpha$ we have the following family of dispersive bounds
		\begin{align*}
			\|e^{-itH} H^{\frac{\gamma-n}{2\alpha}} P_{ac}(H)\|_{L^{1}\to L^{\infty}}\les|t|^{-\frac{\gamma}{2\alpha}}.
		\end{align*}

	\end{corollary}

	This work is inspired by previous work of the first and third authors, \cite{EGWaveOp,EGWaveOp2}, studying the boundedness of the wave operators  when $\alpha=m\in \mathbb N$.  There are several technical challenges in this adaptation, such as the lack of a ``splitting identity" that allows one to explicitly equate the integer order Schr\"odinger resolvents to the more-well known second order resolvent, see \eqref{eqn:Resol} below.
	
	There has been substantial work on the $L^p(\R^n)$ boundedness of the wave operators when $\alpha=m\in \mathbb N$, with recent growth in the literature when $m>1$.  The first higher order result, when $(m,n)=(2,3)$ was established by the second and third authors in 2020, \cite{GG4wave}.  The first and second authors extended the range to $(m,n)$ for all $n>2m$ in \cite{EGWaveOp,EGWaveOp2}.  Mizutani, Wan, and Yao studied the case of $(m,n)=(2,1)$ in \cite{MWY}, and studied the endpoints and effect of threshold resonances in the $(m,n)=(2,3)$ case in \cite{MWY3c,MWYr}.  Galtbayar and Yajima consider the case of $(m,n)=(2,4)$ in \cite{GY}.

	The study of the wave operators when $m>1$ partially built upon work on dispersive estimates. Feng, Soffer, Wu and Yao proved ``local dispersive estimates" on the solution operator as a map between weighted $L^2$-based  in \cite{soffernew}.  The third author and Toprak, \cite{GT4}, along with the first author, \cite{egt}, provided ``global dispersive estimates" considering the solution operator as a map from $L^1$ to $L^\infty$ for the fourth order operator in dimensions $n=4$ and $n=3$ respectively.  The authors recently proved dispersive estimates for scaling-critical potentials when $2m<n<4m$, \cite{EGG wiener}.
	
	The wave operators for the usual Schr\"odinger operator $-\Delta+V$, when $m=1$ are well-studied, beginning with the pioneering works of Yajima, \cite{YajWkp1,YajWkp2,YajWkp3}. Which inspired further work when $m=1$  in all dimensions $n\geq 1$, see \cite{JY2,JY4,DF,Miz, Yaj,YajNew} for example.  On $\mathbb R^3$, Beceanu and Schlag  obtained detailed structure formulas for the wave operators, \cite{Bec,BS,BS2}.  The $L^2$ existence and other properties of the higher order wave operators have been studied by many authors, including Agmon \cite{agmon}, Kuroda \cite{Kur1, Kur2}, H\"ormander \cite{Hor}, and Schechter, \cite{Sche,ScheArb}.   
	
	There has been much interest in non-linear fractional Schr\"odinger equations, see for example \cite{Cnonlin,HSnonlin,GZnonlin,FZnonlin,PSnonlin,Dnonlin,BGVnonlin}, studying well-posedness, blow-up and scattering.
	However, the linear analysis is more limited with   results focusing on the free equation $iu_t=(-\Delta)^su$.  Cho, Ozawa, and Xia studied dispersive and Strichartz estimates for the free operator assuming initial data in distorted Besov spaces, \cite{COX}.  Further study of Strichartz estimates for related operators may be found in \cite{CL,GW}, for example. To the best of our knowledge, the only results on dispersive estimates for a perturbed equation is that of the authors in \cite{EGG frac disp} and a recent paper by Taira, \cite{Taira}, which considers local time decay on weighted $L^2$ spaces for positive potentials.

	Our analysis relies on careful study of the resolvent operators, which are defined by $\mR_V(\lambda)=((-\Delta)^{\alpha}+V-\lambda)^{-1}$ and $\mR_0(\lambda)=((-\Delta)^{\alpha}-\lambda)^{-1}$. 
	Our usual starting point to study the wave operators is the stationary representation 
	\begin{align*}
		W_+u
		&=u-\frac{1}{2\pi i} \int_{0}^\infty \mR_V^+(\lambda) V [\mR_0^+(\lambda)-\mR_0^-(\lambda)] u \, d\lambda,
	\end{align*}
	here the superscripts `+' and `-' denote the usual limiting values as $\lambda$ approaches the positive real line from above and below, \cite{EGG frac disp}.
	Since the identity operator is bounded on $L^p$, we need only bound the second term involving the integral.  It is convenient to make the change of variables $\lambda \mapsto\lambda^{2\alpha }$ and consider the integral kernel of the operator
	\begin{align}\label{eqn:wave op defn}
		-\frac{\alpha}{\pi i} \int_0^\infty \lambda^{{2\alpha-1}}\mR_V^+(\lambda^{2\alpha })V[\mR_0^+-\mR_0^-](\lambda^{2\alpha })\, d\lambda.
	\end{align}
	Our result in Theorem~\ref{thm:small} follows by using resolvent identities to expand $\mR_V^+$ in an infinite series 
	and directly summing the series.  To remove the smallness assumption to show that the operator defined in \eqref{eqn:wave op defn} extends to a bounded operator on $L^p$ requires different strategies in the low ($0<\lambda\ll 1$) and high ($\lambda \gtrsim 1$) energy regimes.  To delineate these cases, we use the even, smooth cut-off function $\chi$ with $\chi(\lambda)=1$ for $|\lambda|<\lambda_0$ for some sufficiently small $\lambda_0\ll 1$, and $\chi(\lambda)=0$ for $|\lambda|>2\lambda_0$, as well as the complimentary cut-off $\widetilde \chi(\lambda)=1-\chi(\lambda)$.
	
	When $\alpha=m\in \mathbb N$, we have the splitting identity for $z\in\C\setminus[0,\infty)$, (c.f. \cite{soffernew})
	\be\label{eqn:Resol}
	\mR_0(z)(x,y):=((-\Delta)^m   -z)^{-1}(x,y)=\frac{1}{ mz^{1-\frac1m} }
	\sum_{\ell=0}^{m-1} \omega_\ell R_0 ( \omega_\ell z^{\frac1m})(x,y)
	\ee
	where $\omega_\ell=\exp(i2\pi \ell/m)$ are the $m^{th}$ roots of unity, $R_0(z)=(-\Delta-z)^{-1}$ is the usual ($2^{nd}$ order) Schr\"odinger resolvent.  For the fractional operators, when $\alpha\notin \mathbb N$, we lack this explicit relationship to the $m=1$ Schr\"odinger resolvents. We instead utilize the representations developed in \cite{EGG frac disp} stated in Proposition~\ref{prop:F} below as well as directly bounding Fourier multipliers corresponding to the Born series in Theorem~\ref{thm:Born} below.

	We note that the different assumptions on the potential we impose based on the size of $n$ versus $\alpha$ in Theorems~\ref{thm:small} and \ref{thm:main} are natural.  
	Smoothness of the potential is required for the integer order Schr\"odinger  operator  in high dimensions since the kernel free resolvent $R_0^\pm(\lambda^2)$ grows like $\lambda^{\frac{n+1}{2}-2m}$ as the spectral parameter $\lambda \to \infty$.  This causes the $L^1\to L^\infty$ dispersive estimates to fail in dimensions greater than $4m-1$ without some measure of smoothness on the potential, see the counterexample constructed by the second author and Visan \cite{GV}, later extended by the authors to the higher order case, \cite{EGG counter}.  In particular, when $n>4m-1$ one can construct a compactly supported potential $V\in C^\alpha(\R^n)$ for all $0\leq \alpha<\frac{n+1}{2}-2m-\frac{n}{p}$ for which the wave operators are unbounded on $L^p(\R^n)$ for $\frac{2n}{n-4m+1}<p\leq \infty$.     As in the integer order analysis, \cite{YajWkp1,EGWaveOp}, we impose a condition on the $\mathcal F L^r$ norm of the potential, which requires sufficient smoothness. The $\epsilon$-smoothness requirement in the case $n=4\alpha -1$ could be an artifact of our methods.

	We assume that zero energy is regular, that is there are no threshold resonances or eigenvalues.  These can be characterized in terms of distributional solutions to $H\psi=0$, with $\psi$ in weighted $L^2(\R^n)$ spaces, see Section 8 of \cite{soffernew} for the integer order case and \cite{EGG frac disp} for the fractional case.  The effect of zero energy resonances or eigenvalues on the $L^p$-boundedness of the integer order wave operators is well-studied.  In the classical $m=1$ case the wave operators are generically bounded on $1<p<\frac{n}{2}$ in the presence of a threshold obstruction when $n\geq 3$, while further orthogonality conditions allows one to obtain a larger range, \cite{YajNew,GGwaveop,YajNew2,YajNew3}.  In the higher order case, $m\in\mathbb N$ and $m>1$, the wave operators are bounded for $1<p<\frac{n}{2m}$ in the presence of zero energy eigenvalues when $n>4m$, \cite{EGL}.  In lower dimensions, there is a more complicated resonance structure, see \cite{CSWY} for odd $n$.  In the case of an eigenvalue only, if the zero energy eigenspace is orthogonal to $x^{\alpha} V(x)$ for multi-indices $|\alpha|<k_0$, the wave operators are bounded on $1\leq p<\frac{n}{2m-k_0}$ and one can recover $p=\infty$ if $k_0>2m$, \cite{EGL2}.  One expects analogous results would hold for the fractional operators in the presences of zero energy obstructions, we plan to address this in a future paper.

	The paper is organized as follows.  In Section~\ref{sec:Born} we begin by analyzing the Born series terms that arise by iterating the resolvent identity for the perturbed resolvent in the stationary representation, \eqref{eqn:wave op defn}.   Next in Section~\ref{sec:low}, we control the contribution of the tail of the Born series in the low energy regime, when the spectral parameter $\lambda$ is in a neighborhood of zero.
	In Section~\ref{sec:Minv}, we provide the technical arguments about inverse operators in the low energy regime to complete the low energy analysis.
	In Section~\ref{sec:high} we control the remainder in the high energy regime, when $\lambda\gtrsim 1$.

	\section{Born Series}\label{sec:Born}

	By iterating the resolvent identity, one has the expansion
	\begin{align}\label{eqn:born identity}
	\mR_V(z)=\sum_{J=0}^{2\ell}\big[ \mR_0(z)(-V\mR_0(z))^J \big]- (\mR_0(z)V)^\ell \mR_V(z) (V\mR_0(z))^\ell.
	\end{align}
	Consider the contribution of an arbitrary summand in the Born series to \eqref{eqn:wave op defn},
	$$
	W_J:=(-1)^{J+1}\frac{1}{2\pi i}\int_0^\infty   (\mR_0^+(\lambda)V)^{J} [\mR_0^+(\lambda)-\mR_0^-(\lambda)]\, d\lambda.
	$$
	In this section, we   modify the   proof in \cite{EGWaveOp}, which was inspired by Yajima's work at  \cite{YajWkp1} for the classical Schr\"odinger,  to control the Born series terms for the fractional Schr\"odinger operators. We prove that $W_J$
	extends to a bounded operator on $L^p(\R^n)$, $1\leq p\leq \infty$:
	\begin{theorem}\label{thm:Born}  Fix $\alpha>1$, a  natural number  $n>2\alpha$, $1\leq p\leq \infty$, and $0<\delta\ll 1$. Then $\exists C=C(\delta,n,\alpha)>0$ so that
		for $2\alpha <n<4\alpha -1$,  we have  
		$$\|W_J\|_{L^p\to L^p}\leq C^J    \| \la \cdot \ra^{ \frac{4\alpha +1-n}{2}+\delta} V(\cdot )\|_{L^2}^J, $$
		for $n=4\alpha -1$, we have 
		$$\|W_J\|_{L^p\to L^p}\leq C^J  \|  \la x\ra^{1+\delta} V\|_{H^{ \delta}}^{J}, $$
		for $n>4\alpha -1$, we have 
		$$\|W_J\|_{L^p\to L^p}\leq C^J  \| \mF(\la x\ra^{ \frac{2{n-4\alpha}}{n-1-\delta}+\delta} V)\|_{L^{\frac{n-1-\delta}{{n-2\alpha}-\delta}}}^J.$$
	\end{theorem}	
	In what follows we will ignore most implicit constants; the factor $C^J$ accounts for their contribution depending on $n, \alpha$.  The small potential result, Theorem~\ref{thm:small}, follows from these inequalities.

  As in \cite{YajWkp1,EGWaveOp}, we bound the adjoint operator $Z_J=W_J^*$ on $L^p$, which for fixed $f\in \mathcal S$  is defined by 
	\be\label{eqn:ZJ defn}
	Z_Jf(x)=\lim_{ \epsilon_1 \to 0^+}\cdots \lim_{\epsilon_J\to 0^+} \lim_{\epsilon_0\to 0^+} Z_{J,\vec \epsilon,\epsilon_0} f(x),
	\ee
	where 
	$$
	Z_{J,\vec \epsilon,\epsilon_0}f (x)
	:= \frac1{2\pi i} \int_\R \big[\mR_0(\lambda-i\epsilon_0) V\mR_0 (\lambda+i\epsilon_1) \cdots V\mR_0 (\lambda+i\epsilon_J) f\big](x) d\lambda.
	$$
 As in \cite{YajWkp1}, it suffices to prove that the limit above exists in $L^p$ and the bounds stated in the theorem hold  for $f\in \mathcal S$ and $\widehat V\in C_0^\infty$.  
Following the steps in page 7 and 8 of  \cite{EGWaveOp}, we write 
\begin{multline}\label{eq:ZJfin}
	Z_Jf(x) \\=\lim_{\epsilon_1\to 0^+} \cdots \lim_{\epsilon_J\to 0^+}  
	\int_{ \R^{n} } T_{k_1,\epsilon_1}^\alpha\bigg\{  \int_{ \R^{n} } T_{k_2,\epsilon_2}^\alpha\bigg\{ \cdots \int_{\R^n}  K_J(k_1,k_2,\dots, k_J) T_{k_J, \epsilon_J }^\alpha f_{k_J} \, dk_J \bigg\}\cdots   \bigg\} dk_2 \bigg\} dk_1,
	\end{multline}
where $  K_J(k_1,k_2,\dots, k_J):= \prod_{j=1}^J \widehat V(k_j-k_{j-1})$ (with $k_0:=0$) and $f_{k_J}(x):=e^{ik_J\cdot x} f(x)$, and 
\be \label{Tmke}
		T_{k,\epsilon}^\alpha f  =\mathcal F^{-1}\bigg( \frac{\widehat f(\xi)}{|\xi-k|^{2\alpha}-|\xi|^{2\alpha}-i\epsilon} \bigg).
	\ee 
	Accordingly, we need to understand the operators $T^\alpha_{k,\epsilon}$.  Let
	\be \label{eqn:pomega def}
	p_\omega(\xi):=\frac{|\xi-\omega|^{2}-|\xi|^2}{|\xi-\omega|^{2\alpha}-|\xi|^{2\alpha}}, \text{ where }
	\omega=\frac{k}{|k|}\in S^{n-1}.
	\ee
	Unlike the case when $\alpha\in \mathbb N$, we cannot neatly factor here.
	We therefore have 
	$$
	T_{k,\epsilon}^\alpha f  
	=\frac1{2i|k|^{2\alpha-1}} \mathcal F^{-1}\bigg( \frac{p_\omega(\xi/|k|)\widehat f(\xi)}{ -\frac{i|k|}2+i\omega\cdot\xi-\frac{\epsilon p_\omega(\xi/|k|)}{2|k|^{2\alpha-1} } } \bigg).
	$$
	It is easy to see that $p_\omega(\xi)\geq 0$, in fact, the proof of Lemma~\ref{lem:gk} below implies that $p_\omega(\xi)\approx \la \xi\ra^{2-2\alpha} > 0$.   Writing 
	$$
	\frac{1}{ -\frac{i|k|}2+i\omega\cdot\xi-\frac{\epsilon p_\omega(\xi/|k|)}{2|k|^{2\alpha-1} } }=-\int_0^\infty e^{-\frac{i|k|t}2+ it\omega\cdot\xi} e^{-\frac{\epsilon p_\omega(\xi/|k|)}{2|k|^{2\alpha-1} }t} dt,
	$$	
	we obtain
	$$
	\mathcal F^{-1}\bigg( \frac{p_\omega(\xi/|k|) \widehat f(\xi)}{ -\frac{i|k|}2+i\omega\cdot\xi-\frac{\epsilon p_\omega(\xi/|k|)}{2|k|^{2\alpha-1} }} \bigg)(x)=-\int_0^\infty e^{ -\frac{i|k|t}2 } h_{k,\frac{\epsilon t}{2|k|^{2\alpha-1}}} * f (x+t\omega) dt,
	$$
	where $*$ denotes convolution and
	$$
	h_{k,\epsilon  } =\mF^{-1}\Big(p_\omega(\xi/|k|) e^{-  \epsilon   p_\omega(\xi/|k|)  }\Big).
	$$
	With this notation, we have 
	$$
	T_{k,\epsilon}^\alpha f(x) =\frac{i}{2|k|^{2\alpha-1}}\int_0^{\infty} \int_{ \R^{ n} } e^{-i|k|t/2} h_{k,\frac{\epsilon t}{2|k|^{2\alpha-1}}}(y ) f(x-y +t\omega)  \, dy  \, dt.
	$$
	To study the limit as $\epsilon\to 0^+$, we need the following lemma: 
	\begin{lemma}\label{lem:gk}
		
		We have the following bounds  (with $k=s\omega, s>0, \omega\in S^{n-1}$)
		$$
		\big\|  \sup_{\epsilon>0} h_{k,\epsilon } \big\|_{L^1}\les 1,
		$$
		$$
		\big\|\sup_{\epsilon>0} |\partial_s^j h_{s\omega,\frac{\epsilon }{ s^{2\alpha-1}}}| \big\|_{L^1}\les  s^{-j},\,\,\, j= 1,2. 
		$$
		Furthermore, $h_{k,\epsilon}$ converges to $h_k:=h_{k,0}$ and $\partial_s^j h_{s\omega,\frac{\epsilon }{ s^{{2\alpha-1}}}} $  converges to $\partial_s^j h_k  $ as $\epsilon\to 0$ a.e. and in $L^1$, and $h_k$ satisfies the same bounds above.    
	\end{lemma} 
	To prove this lemma, we need the following lemmas from \cite{EGG frac disp}:
	\begin{lemma}\label{lem:small xi}
		If $g$ compactly supported on $\R^n$, and is smooth on $\mathbb R^n\setminus\{0\}$ with $|\nabla^Ng(\xi)|\les |\xi|^{\gamma-N}$ for some $\gamma>-n$ and $N=0,1,\dots$ for $\xi\neq 0$. Then
		$|\nabla^j\widehat g(x)|\les \la x\ra^{-n-\gamma-j}$, j=0,1,2,...  In particular, $\widehat g\in L^1$ if $\gamma>0$.
	\end{lemma} 
	\begin{lemma}\label{lem:FT}
	
	Let $g$ be a smooth function, supported  away from zero on $\mathbb R^n$, that satisfies $|\nabla^N g(\xi)|\les |\xi|^{\gamma-N}$ for some $\gamma<0$ and $N=0,1,2,\dots$. Then $\widehat{g}$ is a smooth function on $\mathbb \R^n\setminus\{0\}$ satisfying 
	$$
		|\nabla^j\widehat{g}(x)|\les  \left\{
		\begin{array}{ll}
			|x|^{-\gamma-n-j} & \text{if }\gamma+j>-n,\\
			\la \log |x|\ra & \text{if }\gamma+j=-n,\\
			1 & \text{if }\gamma+j<-n.
		\end{array}
		\right. 
	$$
	Morever for $|x|\gtrsim 1$, $|\nabla^j\widehat{g}(x)|\les |x|^{-M}$ for all $M,j$.

\end{lemma}
	\begin{proof}[Proof of Lemma~\ref{lem:gk}] We first prove the claims for $h_k$.
		Note that 
		$$
			h_{s\omega}(x)=s^n \mF^{-1} p_{\omega}(xs)=s^n h_{\omega}(sx)
		$$	
		Therefore $\|h_{s\omega}\|_{L^1}=\|h_\omega\|_{L^1}$ and we may take $s=1$.
		Without loss of generality, we also assume that $\omega=e_1$.
		We decompose $p_{e_1}$ into three pieces, when $\xi$ is near zero, near $e_1$, and away from both.  We write
		$p_{e_1}=p_1+p_2+p_3$ respectively defined by smooth cut-offs and define $h_i=\mF^{-1}p_i$.
		
		First, we consider $p_1(\xi)=p_{\omega}(\xi)\chi(100\xi)$ and write
		\be\label{eqn:phi g}
			p_1(\xi):=\phi(\xi)+g(\xi), \quad \text{where} \quad
			\phi(\xi)=\chi(100\xi)
			 \frac{1-2\xi_1}{|\xi-e_1|^{2\alpha}}.
		\ee
		Note that  $\phi\in\mathcal S(\R^n)$, and 
		$$
			g(\xi)=\bigg[\frac{1-2\xi_1}{|\xi-e_1|^{2\alpha}-|\xi|^{2\alpha}}-\frac{1-2\xi_1}{|\xi-e_1|^{2\alpha}}\bigg]\chi(100\xi),
		$$
		is easily seen to satisfy the hypotheses of Lemma~\ref{lem:small xi}.  So that 
		$$
			|\mathcal F[\chi(100 \cdot)p_{\omega}(\cdot)](x)|=|h_1(x)|\les \la x\ra^{-n-2\alpha}.
		$$
		Moreover, analogous bounds hold for $p_2(\xi)=\chi(100|\xi-\omega|)p_{\omega}(\xi)$ and hence $h_2$.
		
		We define $\chi_3(\cdot)=1-\chi(100 \cdot)-\chi(100|\cdot-\omega|)$, and consider $p_3(\xi)=\chi_3(\xi)p_{\omega}(\xi)$.  Here we note that
		$$
			p_{3}(\xi)=\chi_3(\xi)
			\frac{1-2\xi_1}{(|\xi|^2+1-2\xi_1)^{\alpha}-|\xi|^{2\alpha}}=\chi_3(\xi)|\xi|^{2-2\alpha} \eta\bigg(\frac{1-2\xi_1}{|\xi|^2} \bigg),
		$$
		where
		$$
			\eta(z)=\frac{z}{(1+z)^\alpha-1}\chi_{[-1+c,C]}(z).
		$$
		Here $\chi_{[-1+c,C]}(z)$ is smooth cut-off to the interval $[-1+c,C]$ for some $c,C>0$.  Note that, on the support of $\chi_3(\xi)$, we have $\frac{1-2\xi_1}{|\xi|^2}\in [-1+c,C]$.  Since $\eta$ is analytic and bounded in an open neighborhood of this interval in the complex plane (the singularity at $z=0$ is removable), $\eta$ has bounded derivatives to arbitrary order.\footnote{Also note that $\eta\gtrsim 1$ on the interval $ [-1+c,C]$, which implies that on the support of $\chi_3$, $p_\omega(\xi)\approx |\xi|^{2-2\alpha}$. Observing that $p_w \approx 1$ on the support of $\chi_1+\chi_2$ implies that  $p_\omega(\xi)\approx \la\xi\ra^{2-2\alpha}$.}  Using the chain rule, we see that
		$$
			|\nabla^N p_{3}(\xi)|\les \la \xi \ra^{2-2\alpha-N}.
		$$
		Therefore, using Lemma~\ref{lem:FT}, we conclude that 
		$$
		h_3(x)=\mF^{-1} (  p_{3} )(x)=O\big(\min(|x|^{-n-1},|x|^{-n+2\alpha-2})\big),
		$$
  which implies that 
		$h_3\in L^1$ (since $\alpha>1$).  This yields the claim for $j=0$.

		For $j>0$, note that 
		$$
		\partial_s\mF^{-1}p_3(sx)=x\cdot [\nabla\mF^{-1} p_{3}](xs)=\frac1s \mF^{-1}(\nabla\cdot\xi\,  p_{3}(\xi)) (xs).
		$$
		Similarly, $(s\partial_s)^\ell \mF^{-1}p_3(sx)=  \mF^{-1}((\nabla\cdot\xi)^\ell p_3(\xi)) (xs)$. Therefore,  
		$$
		|\partial_s^j s^{n} \mF^{-1}p_3(sx)|\les \sum_{\ell=0}^j s^{n+\ell-j} s^{-\ell}  | \mF^{-1}((\nabla\cdot\xi)^\ell p_3(\xi)) (xs)|.
		$$
		The claim for $h_3=\mF^{-1}p_3$ follows from this as above since 	$(\nabla\cdot\xi)^\ell p_3(\xi)$ satisfies the same bounds as $p_3(\xi)$. 
		
		We now turn to $p_1$, and the proof follows as above since $(\nabla \cdot \xi)^\ell g$ satisfies the same bounds as $g$.  For $p_2$, we write
		$$
			p_2(\xi)=\chi(100|\xi-e_1|)
			\frac{2\xi_1-1}{|\xi|^{2\alpha}}+g_2(\xi),
		$$
		and 
		$$
			g_2(\xi)=\bigg[\frac{1-2\xi_1}{|\xi-e_1|^{2\alpha}-|\xi|^{2\alpha}}-\frac{2\xi_1-1}{|\xi|^{2\alpha}}\bigg]\chi(100|\xi-e_1|).
		$$
		Now,
		$(\nabla \cdot \xi)^{\ell}g_2(\xi)$ 		satisfies the hypotheses of Lemma~\ref{lem:small xi} (centered at $\xi=e_1$ instead of zero) for $\ell\leq 2$ and $\gamma=2\alpha-2>0$.
		
		Now, we consider $h_{k,\epsilon}$.  
		Let $H_\omega(\epsilon,x)=\mF^{-1}\Big(p_\omega e^{-  \epsilon   p_\omega  }\Big)(x)$. 	
		We first consider 
		$$
			H_3(\epsilon,x):=\mF^{-1}\Big(\chi_3 p_\omega e^{-  \epsilon   p_\omega  }\Big)(x)=\mF^{-1}\Big(p_3 e^{-  \epsilon   p_\omega  }\Big)(x).
		$$
		Using the bounds on the derivative of $p_3$ and noting that 
		$\sup_{\alpha\geq 0} \alpha^N e^{-\alpha}\les 1$ for any $N\geq 0$, and that $|\nabla^j p_\omega|\les |\nabla^j p_3|$ on the support of $\chi_3$, and that $0\leq p_3\leq p_\omega$, we conclude that
		$$
		\big|\nabla_\xi^N  [p_3 (\xi) e^{-  \epsilon   p_\omega(\xi ) }]\big| \les \frac{1}{\la \xi \ra^{2\alpha-2+N}}, \qquad N=0,1,2,\dots
		$$
		Therefore we have  
		\be\label{eqn:homega bound}
			|H_3(\epsilon,x)|=|\mF^{-1} ( p_3 e^{-  \epsilon   p_\omega  } )(x)|\les    \min(|x|^{-n-1},|x|^{-n+2\alpha-2}) ,
		\ee
		uniformly in  $\epsilon>0$.
		This yields the claim for $j=0$ for the contribution of $H_3$ to $h_{k,\epsilon}=s^nH_\omega(\epsilon,sx)$.   
		
		We now turn to 
		$$
			H_1(\epsilon,x):=\mF^{-1}\Big(\chi_1 p_\omega e^{-  \epsilon   p_\omega  }\Big)(x)=\mF^{-1}\bigg(p_1e^{-\epsilon \widetilde p_1}\bigg)(x), 
		$$
		where $\widetilde p_1(\xi)=\chi(10\xi)p_\omega(\xi)$.
		Using \eqref{eqn:phi g}, we have $p_1=\phi+g$.  Defining $\widetilde \phi, \widetilde g$ analogously, we have 
		$0\leq \phi \leq \widetilde\phi$ and $0\leq g\leq \widetilde g$.  
		So that
		$$
			p_1e^{-\epsilon p_\omega}=\phi e^{-\epsilon\widetilde \phi}+\phi e^{-\epsilon\widetilde \phi}(e^{-\epsilon \widetilde g}-1) +ge^{-\epsilon\widetilde g}e^{-\epsilon\widetilde \phi}.
		$$
		The last two summand satisfy the hypotheses of Lemma~\ref{lem:small xi} while the first summand is in $\mathcal S(\R^n)$ with uniform in $\epsilon$ bounds on the derivatives.  Therefore,
		$$
			|H_1(\epsilon,x)|\les \la x\ra^{-n-2\alpha}
		$$
		uniformly in $\epsilon>0$.  A similar argument for $p_2$ yields the same bounds for $H_2(\epsilon,x)$.  Therefore, we conclude that 
		$$
			\sup_{\epsilon\geq 0} |H_{\omega}(\epsilon,x)|\les \min(|x|^{-n-1},|x|^{-n+2\alpha-2}),
		$$
		which yields the claim.
		
		Similarly, note that 
		$$\big|\nabla_\xi^N  [p_3 (\xi) (e^{-  \epsilon   p_\omega(\xi ) }-1) ]\big|\les \frac{\epsilon }{\la \xi \ra^{4\alpha-4+N}}, \qquad N=0,1,2,\dots.
		$$
		This implies the a.e. and $L^1$ convergence of the contribution of $H_3$ in $h_{k,\epsilon}$  to $h_k $.
		
		For $H_1$ we write 
		$$ 
			p_1 (\xi) (e^{-  \epsilon   p_\omega(\xi ) }-1)=
			\phi(\xi)e^{-\epsilon\widetilde \phi(\xi)}\bigg( 
			e^{-\epsilon \widetilde g(\xi)}-1			\bigg)
			+ \phi(\xi)\bigg(e^{-\epsilon \widetilde \phi(\xi)}-1\bigg)+g(\xi)\bigg(
			e^{-\epsilon \widetilde p_1(x)}-1
			\bigg).
		$$
		The first and third summands satisfy the hypotheses of Lemma~\ref{lem:small xi} with an additional factor of $\epsilon$.  The second summand is in $\mathcal S(\R^n)$ with all derivatives bounded by $\epsilon$.  A similar argument applies for the contribution of $H_2$, hence $h_{k,\epsilon}$ converges a.e. and in $L^1$ to $h_k$.
		
		For the $j$th derivative of $h_{k,\epsilon}$, by chain rule and scaling as above, it suffices to prove that the $L^1$ norms of $\sup_{\epsilon} \epsilon^{j_1}\partial_\epsilon^{j_1}  (x\cdot\nabla_x)^{j_2} \mF^{-1}[p_{\omega}e^{-\epsilon p_\omega}](x)$ are $\les 1 $ for $j_1,j_2\geq 0$ with $j_1+j_2\leq 2$. Noting that
		$|\epsilon^{j_1}\partial_{\epsilon}^{j_1}e^{-\epsilon p_\omega}|= |(-\epsilon p_\omega)^{j_1}e^{-\epsilon p_\omega}|\les e^{-\epsilon p_\omega/2}$, the arguments above remain valid.   Convergence of the $s$ derivatives of  $h_{k,\epsilon}$ follow similarly. 
	\end{proof}

	Using Lemma~\ref{lem:gk} and dominated convergence theorem, we conclude that for $f\in\mathcal S$
	and for all $x\in \R^n$,  
	$$	\lim_{\epsilon\to 0^+} T_{k,\epsilon}^\alpha f(x)= \frac{i}{2|k|^{2\alpha-1}} \int_0^\infty e^{-it|k|/2} \int_{\R^n}h_k(y )f(x-y +t\omega) \, dy  \, dt:= T_k^\alpha f(x). $$
	Following the notation of \cite{YajWkp1}, for $\epsilon>0$, let 
	$$
	G_\epsilon f=\int_{\R^n} T_{k,\epsilon}^\alpha f(k,\cdot) dk,\qquad   G_0 f  =\int_{\R^n} T_{k}^\alpha f (k,\cdot) dk,
	$$
	Note that
	\be\label{eq:geps}
		G_\epsilon f(x)=\int_{\R^n}\frac{i}{2|k|^{2\alpha-1}}\int_0^{\infty} \int_{ \R^{ n} } e^{-i|k|t/2}  h_{k,\frac{\epsilon t}{2|k|^{2\alpha-1}}}(y ) f(k,x-y +t\omega)  \, dy  \, dt\,dk.
	\ee
	Passing to polar coordinates, $k=s\omega$, and changing the order of integration, we have 
	$$
		G_\epsilon f(x)=\frac{i}2\int_{S^{n-1}} \int_0^{\infty} 	F_\epsilon( t,\omega,x)  \, dt\,d\omega, 
	$$
	where
	$$
		F_\epsilon( t,\omega,x)= \int_0^\infty  e^{-is t/2} s^{n-2\alpha}  	h_{s\omega,\frac{\epsilon t}{2s^{2\alpha-1}}}* f(s\omega, \cdot)(x +t\omega ) \, ds.
	$$
	Also note that $G_0f$   satisfies the same formula  with  $F_0 $  replacing  $F_\epsilon $.

	\begin{lemma}\label{lem:Geps} 
		Let  $\epsilon>0$, $1\leq p \leq \infty$,  and  $f(k,x)\in \mathcal S(\R^n_k,\mathcal S(\R^n_x))$. For all  $n>2\alpha+1$, we have
		$$ \| G_\epsilon f\|_{L^p}  \leq C_{n,\alpha} \int_{\R^{n}}  
		\la k\ra^{ 2 }  
		\sum_{j =0}^2   \|D_k^{j }f(k,\cdot)\|_{L^p}  \frac{dk}{|k|^{1+2\alpha}}.
		$$
		For $2\alpha<n\leq 2\alpha+1$, we have 
		$$ \| G_\epsilon f\|_{L^p}  \leq C_{n,\alpha} \int_{\R^{n}}  
		\la k\ra^{\frac{n}2-\alpha+1}  
		\sum_{j =0}^3   \|D_k^{j }f(k,\cdot)\|_{L^p}  \frac{dk}{|k|^{\frac{n}2+\alpha}}.
		$$ 
		Moreover, $ G_\epsilon f\to G_0 f$  in $L^p$ as $\epsilon \to 0^+$.
	\end{lemma}
	
	We note that the exponents here on the factor of $|k|$ in the denominator differ slightly from the analogous bounds when $\alpha=m\in \mathbb N$ in Lemma~2.3 of \cite{EGWaveOp}.  This is the result of not bounding positive powers of $s=|k|$ in the proof below $\langle k \rangle$, this choice does not affect the subsequent analysis.
	
	\begin{proof}
		Note that
		$$\big\|  F_\epsilon( t,\omega,x) \big\|_{L^p_x} \les    \int_0^\infty s^{n-2\alpha}  \| \sup_\epsilon h_{s\omega,\epsilon}\|_{L^1}
		\|f(s\omega, \cdot)\|_{L^p} \, ds  \\ \les     \int_0^\infty s^{n-2\alpha}  
		\|f(s\omega, \cdot)\|_{L^p} \, ds,
		$$
		which suffices for the integral in $0<t\les 1$. 
		For $t\gg 1$ and $n>2\alpha+1$,   we integrate by parts  twice in the $s$ integral to obtain 
		$$  |F_\epsilon( t,\omega,x) | \les  \frac{1}{t^2}\int_{\R^ n}\int_0^\infty \big|   \partial_s^2 \big( s^{n-2\alpha} h_{s\omega,\frac{\epsilon t }{2 s^{2\alpha-1}}}(y ) f(s\omega,x-y +t\omega)\big)\big|  \, ds\, dy. 
		$$ 
		Let $H_{s\omega }(y)=|\sup_{\epsilon>0, j=0,1,2  } s^j \partial_s^j  h_{s\omega,\frac{\epsilon t }{2 s^{2\alpha-1}}}(y )|$.  Using this we obtain the bound 
		\begin{multline*}  
		|F_\epsilon( t,\omega,x) | \les  
		\frac{1}{t^2}\int_{\R^n}\int_0^\infty \la s\ra^2 s^{n-2\alpha-2}     H_{s\omega}(y) \sum_{j=0}^2 \big|   \partial_s^j f(s\omega,x-y +t\omega) \big|  \, ds\, dy \\
		\les  \frac{1}{t^2} \int_{ \R^n } \int_0^\infty  H_{s\omega}(y) \la s\ra^2 s^{n-2\alpha-2}    \sum_{j=0}^2 \big|   \partial_s^j f(s\omega,x-y +t\omega) \big|  \, ds\, dy. 
		\end{multline*}
		
		By Lemma~\ref{lem:gk}, $\|H_{s\omega}\|_{L^1}\les 1$, therefore uniformly in $t$ and $\omega$, we have  
		$$\big\|  F_\epsilon( t,\omega,x) \big\|_{L^p_x} \les \frac1{\la t\ra^2}
		\int_0^\infty \la s\ra^2 s^{n-2\alpha-2}       \sum_{j=0}^2 \big\|   \partial_s^j f(s\omega,\cdot)\big\|_{L^p}  \, ds,
		$$
		which implies the claim for $G_\epsilon f$ when $n>2\alpha+1$. The convergence of $G_\epsilon f$ to $G_0f$ in $L^p$ also follows by applying the same argument with $ h_{s\omega,\frac{\epsilon t}{2s^{2\alpha-1}}}-h_{s\omega}$ replacing $ h_{s\omega,\frac{\epsilon t}{2s^{2\alpha-1}}} $ and using dominated convergence theorem. 
		
		We now consider the case $2\alpha< n \leq 2\alpha+1$ and $t\gg 1$. After an integration by parts, we have
		$$ F_\epsilon( t,\omega,x)   =-\frac{2i}t \int_0^\infty  e^{-is t/2} \partial_s[s^{n-2\alpha}  h_{s\omega,\frac{\epsilon t}{2s^{{2\alpha-1}}}}* f(s\omega, \cdot)(x +t\omega ) ] \, ds. 
		$$
		We cannot integrate by parts again to gain another power of $t$. Therefore we utilize the identity (with $ K(s) =\partial_s[s^{n-2\alpha}  h_{s\omega,\frac{\epsilon t}{2s^{{2\alpha-1}}}}* f(s\omega, \cdot)(x +t\omega ) ]$)
		$$
		\int_0^\infty  e^{-is t/2} K(s) ds= \frac12 \int_0^{2\pi/t}  e^{-is t/2} K(s) ds + \frac12  \int_{0}^\infty  e^{-i (s+2\pi/t) t  /2  } [K(s+2\pi/t) -K(s )] ds.
		$$ 
		This implies that (with $\eta=\frac{n}2-\alpha\in(0,\frac12]$)
		\begin{multline*}
		\Big\|\int_0^\infty  e^{-is t/2} K(s) ds\Big\|_{L^p_x}
		\les \\ \int_0^{2\pi/t}  \|K(s)\|_{L^p_x} ds+
		\int_{0}^\infty  (\|K(s+2\pi/t)\|_{L^p_x}+\|K(s)\|_{L^p_x})^{1-\eta} \Big(\int_{s }^{s+2\pi/t} \big\|\partial_\rho K(\rho) \big\|_{L^p_x} d\rho\Big)^\eta ds\\
		\les t^{-1} \sup_{0<s<1} \|K(s)\|_{L^p_x}+t^{-\eta}\int_0^\infty \big[\sup_{s<\rho<s+1} \|K(\rho)\|_{L^p}\big]^{1-\eta} \big[\sup_{s<\rho<s+1} \|\partial_\rho K(\rho)\|_{L^p}\big]^{\eta} ds. 
		\end{multline*}
		Note that
		$$
		\|K(\rho)\|_{L^p_x} \les \la \rho\ra \rho^{n-2\alpha-1}  \big(\|f(\rho\omega,\cdot)\|_{L^p} +\|\partial_\rho f(\rho\omega,\cdot)\|_{L^p}\big)
		$$
		$$
		\big\|\partial_\rho K(\rho) \big\|_{L^p_x} \les \la \rho\ra^2  \rho^{n-2\alpha-2 }    \big(\|f(\rho\omega,\cdot)\|_{L^p} +\|\partial_\rho f(\rho\omega,\cdot)\|_{L^p}+\|\partial^2_\rho f(\rho\omega,\cdot)\|_{L^p}\big).
		$$
		Therefore, for $t\gg 1$
		$$
		\Big\|\int_0^\infty  e^{-is t/2} K(s) ds\Big\|_{L^p_x}\les t^{-\eta}
		\int_0^\infty \la s\ra s^{n-2\alpha-1} (s^{-1}\la s\ra)^{\eta} \sup_{s<\rho<s+1} \sum_{j=0}^2 \|\partial_\rho^j f(\rho \omega,\cdot )\|_{L^p} ds.
		$$
		Noting that, for $s<\rho<s+1$ 
		$$
		\sum_{j=0}^2 \|\partial_\rho^j f(\rho \omega,\cdot )\|_{L^p} \leq \sum_{j=0}^2 \|\partial_s^j f(s \omega,\cdot )\|_{L^p} +\int_s^{s+1} \sum_{j=0}^3 \|\partial_\rho^j f(\rho \omega,\cdot )\|_{L^p}\, d\rho,
		$$
		and applying Fubini's theorem yield the claim bounding $G_\epsilon$ in $L^p$. Convergence in $L^p$ follows similarly. 
	\end{proof}

	We now return to the operator $Z_J$ defined in \eqref{eq:ZJfin}.
	Taking limits as $\epsilon_j\to 0$ using the lemmas above and tracing the steps in pages 13 and 14  of \cite{EGWaveOp}, we bound $Z_j$ defined in \eqref{eq:ZJfin} as
	$$
	\|Z_Jf\|_{L^p}\les \|F  \|_{L^1((S^{n-1}\times\R^n\times\R)^{J})}\| f\|_{L^p}, 
	$$
	where 
	\begin{multline*}
	F =F (\omega_{ 1},y_{ 1},t_{ 1}, \dots,\omega_J,y_J,  t_J)\\
	:=
	\int_{(0,\infty)^J} \prod_{j= 1}^J \big[s_j^{n-2\alpha} e^{- i\frac{ s_jt_j}2  }  h_{s_j\omega_j} (y_j )\big]   K_J(s_{ 1}\omega_{ 1}, \dots,  s_J\omega_J)   \,ds_J \cdots ds_{ 1},
	\end{multline*}
	where  	$  K_J(k_1,k_2,\dots, k_J)= \prod_{j=1}^J \widehat V(k_j-k_{j-1})$.
	
	The following lemma finishes the proof of $L^p$ boundedness of $Z_J$.
	\begin{lemma} For $2\alpha<n<4\alpha-1$, we have  
		$$\|F \|_{L^1((S^{n-1}\times\R^n\times\R)^{J})}\leq C^J    \| \la \cdot \ra^{ \frac{4\alpha+1-n}{2}+} V(\cdot )\|_{L^2}^J, $$
		for $n=4\alpha-1\in\N$, we have 
		$$\|F \|_{L^1((S^{n-1}\times\R^n\times\R)^{J})}\leq C^J  \|  \la x\ra^{1+} V\|_{H^{0+}}^{J}, $$
		for $n>4\alpha-1$ and $\sigma>\frac{{n-2\alpha}}{n-1}$, we have 
		$$\|F \|_{L^1((S^{n-1}\times\R^n\times\R)^{J})}\leq C^J  \| \mF(\la x\ra^{2\sigma} V)\|_{L^{\frac{n-1}{{n-2\alpha}}-}}^J.$$
		Here $C$ depends on $n,\alpha$ and the actual values of $\pm$ signs. 
	\end{lemma}
	\begin{proof}
		We write $F$ as a sum of $2^{J}$ operators of the form (for each subset $\mJ$ of $\{1,2,...,J\}$)
		$$
		F_{\mJ}  (\omega_{ 1},y_{ 1},t_{ 1}, \dots,\omega_J,y_J,  t_J)= F (\omega_{ 1},y_{ 1},t_{ 1}, \dots,\omega_J,y_J,  t_J) \big[\prod_{j\in\mJ}
		\chi(y_j) \big]  \big[\prod_{j\not \in\mJ}
		\widetilde\chi(y_j) \big].
		$$
		It suffices to prove that each $F_\mJ$ satisfies the claim.
		
		Fix $r\geq 2$ and $\frac1q+\frac1r=1$. By Hausdorff-Young inequality,  we have (with $L^p(\Omega)L^q(D)=L^P(\Omega,L^q(D))$)
		\begin{multline*}\|F_{\mJ}\|_{L^1 (S^{n-1}\times\R^n)^{J} L^r(\R^{J})}\les  
		\int_{(S^{n-1}\times\R^n)^{J}}  \Big[\int_{(0,\infty)^{J} } \big[\prod_{j=1}^J
		s_j^{n-2\alpha}   h_{s_j\omega_j} (y_j )   \big]^q    \times\\ |  K_J(s_{ 1}\omega_{ 1}, \dots,  s_J\omega_J)|^q ds_1\dots ds_J\Big]^{1/q}\big[\prod_{j\in\mJ}
		\chi(y_j) \big]  \big[\prod_{j\not \in\mJ}
		\widetilde\chi(y_j) \big]d\vec y d\vec \omega.
		\end{multline*}
		Note that, by \eqref{eqn:homega bound} in the proof of Lemma~\ref{lem:gk} above (for $0<\delta\ll 1$)
		$$|h_{s\omega}(y)|\les s^n\min((s|y|)^{-n-\delta}, (s|y|)^{-n+\delta})  \les \chi(y)  |y|^{-n+\delta} s^\delta  + \widetilde \chi(y)  |y|^{-n-\delta} s^{-\delta}.
		$$
		Since $\chi(y)  |y|^{-n+\delta} \in L^1$ and $\widetilde \chi(y)  |y|^{-n-\delta}\in L^1$ for any $\delta>0$, we can bound the norm above by  
		$$\int_{(S^{n-1})^{J}}  \Big[\int_{(0,\infty)^{J}}\big[\prod_{j\in\mJ}^J s_j^{(n-2\alpha+\delta)q}    \big] 
		\big[\prod_{j\not \in\mJ}^J s_j^{(n-2\alpha-\delta)q}    \big]  |  K_J(s_{ 1}\omega_{ 1}, \dots,  s_J\omega_J)|^q d\vec s \Big]^{1/q} d\vec \omega.
		$$
		By Holder in $\omega_j$ integrals we conclude that
		\begin{multline} \label{eq:FL1Lr}
		\|F_\mJ \|_{L^1 (S^{n-1}\times\R^n)^{J} L^r(\R^{J})}\les\Big[\int_{\R^{nJ}}    \big[\prod_{j\in\mJ}^J|k_j|^{(n-2\alpha+\delta)q-n+1}     \big] 
		\big[\prod_{j\not \in\mJ}^J |k_j|^{(n-2\alpha-\delta)q-n+1}    \big]  \times \\  |  K_J(k_1, \dots,  k_J )|^q dk_1\dots dk_J\Big]^{1/q}.  
		\end{multline} 
		Similarly, (here $\alpha_j=0$  or $1$ independently) 
		\begin{multline*}
		\|t_1^{\alpha_1}\dots t_J^{\alpha_J}F_\mJ\|_{L^1 (S^{n-1}\times\R^n)^{J} L^r(\R^{J})}\les
		\\\int_{(S^{n-1}\times\R^n)^{J}}  \Big[\int_{(0,\infty)^{J} }\Big|\partial_{s_1}^{\alpha_1}\dots \partial_{s_J}^{\alpha_J}\prod_{j= 1}^J \big(s_j^{{n-2\alpha}}    h_{s_j\omega_j} (y_j ) \big)   \times   K_J(s_{ 1}\omega_{ 1}, \dots,  s_J\omega_J)\Big|^q ds_1\dots ds_J\Big]^{1/q}\\ \times  \big[\prod_{j\in\mJ}
		\chi(y_j) \big]  \big[\prod_{j\not \in\mJ}
		\widetilde\chi(y_j) \big]d\vec y d\vec \omega.
		\end{multline*}
		Since $\partial_s h_{s\omega}$ satisfies the same bounds as $\frac1s h_{s\omega}$, proceeding as above, we obtain the  estimate  
		\begin{multline*}\|t_1^{\alpha_1}\dots t_J^{\alpha_J}F_\mJ \|_{L^1 (S^{n-1}\times\R^n)^{J} L^r(\R^{J})}\les 
		\Big[\int_{\R^{nJ}}  \big[\prod_{j\in\mJ} |k_j|^{({n-2\alpha}+\delta)q-n+1}     \big] 
		\big[\prod_{j\not \in\mJ}  |k_j|^{({n-2\alpha}-\delta)q-n+1}    \big]  \times\\ \Big|  \prod_{j=1}^J(\nabla_{k_j}^{\alpha_j}+|k_j|^{-\alpha_j })  K_J(k_1, \dots,  k_J )\Big|^q  dk_1\dots dk_J\Big]^{1/q}.  
		\end{multline*}
		Using Hardy's inequality, this implies that  
		\begin{multline}\label{eq:alphaFL1Lr}\|t_1^{\alpha_1}\dots t_J^{\alpha_J}F_\mJ  \|_{L^1 (S^{n-1}\times\R^n)^{J} L^r(\R^{J})}\les 
		\Big[\int_{\R^{nJ}}  \big[\prod_{j\in\mJ} |k_j|^{({n-2\alpha}+\delta)q-n+1}     \big] 
		\big[\prod_{j\not \in\mJ}  |k_j|^{({n-2\alpha}-\delta)q-n+1}    \big]  \times\\ \Big|  \prod_{j=1}^J \nabla_{k_j}^{\alpha_j}  K_J(k_1, \dots,  k_J )\Big|^q  dk_1\dots dk_J\Big]^{1/q}.  
		\end{multline}
		Let $2\alpha<n<4\alpha-1$. Applying \eqref{eq:FL1Lr} with $0< \delta<\frac{4\alpha-1-n}2$ and $q=r=2$, we obtain 
		$$
		\|F_\mJ\|^2_{L^1 (S^{n-1}\times\R^n)^{J} L^2(\R^{J})}\les  \int_{\R^{nJ}}   \big[\prod_{j\in\mJ} |k_j|^{ {n-4\alpha}+1+2\delta}     \big] 
		\big[\prod_{j\not \in\mJ}  |k_j|^{ {n-4\alpha}+1-2\delta }    \big]    |  K_J(k_1, \dots,  k_J )|^2 d\vec k.  
		$$
		Note that by Hardy's inequality the integral in $k_J$ is bounded by 
		$$
		\int | |D_{k_J}|^{\frac{4\alpha-1-n}{2}\pm\delta} \widehat V(k_{J-1}-k_J)|^2 dk_J\les   \| \la \cdot \ra^{\frac{4\alpha-1-n}{2}\pm\delta} V(\cdot )\|_{L^2}^2\les \| \la \cdot \ra^{\frac{4\alpha-1-n}{2}+\delta} V(\cdot )\|_{L^2}^2. 
		$$
		Repeated application of this inequality yields
		$$
		\|F_\mJ\|_{L^1 (S^{n-1}\times\R^n)^{J} L^2(\R^{J})}\les    \| \la \cdot \ra^{\frac{4\alpha-1-n}{2}+\delta} V(\cdot )\|_{L^2}^J.
		$$
		Similarly, applying \eqref{eq:alphaFL1Lr} with $r=q=2$  and  $0<\delta\ll 1$ yield
		$$
		\|t_1^{\alpha_1}\dots t_J^{\alpha_J}F_\mJ  \|_{L^1 (S^{n-1}\times\R^n)^{J} L^2(\R^{J})} \les     \| \la \cdot \ra^{2+\frac{4\alpha-1-n}{2}+\delta} V(\cdot )\|_{L^2}^J.
		$$
		We lose two powers here as derivatives in $k_{j-1} $ and $k_j$  may both hit $\widehat V(k_{j-1}-k_j)$.
		Writing 
		$$\prod_{j=1}^J (1+|t_j|) = \sum_{\alpha_1,\dots, \alpha_J\in\{0,1\}} |t_1^{\alpha_1}\dots t_J^{\alpha_J}|,$$ these inequalities imply with that 
		$$\big\|\prod_{j=1}^J \la t_j\ra F_\mJ  \big\|_{L^1 (S^{n-1}\times\R^n)^{J} L^2(\R^{J})}\les  \| \la \cdot \ra^{2+\frac{4\alpha-1-n}{2}+\delta} V(\cdot )\|_{L^2}^J,
		$$
		which by multilinear complex interpolation leads to 
		$$\big\|\prod_{j=1}^J \la t_j\ra^{\frac12+} F_\mJ \big\|_{L^1 (S^{n-1}\times\R^n)^{J} L^2(\R^{J})}\les  \| \la \cdot \ra^{1+\frac{4\alpha-1-n}{2}+\delta+} V(\cdot )\|_{L^2}^J.
		$$
		This proves the claim for $n<4\alpha-1$ by Cauchy-Schwarz in the $t$ integrals.
		
		For $n=4\alpha-1\in\N$,  with $q=2-$, $r=2+$,  \eqref{eq:FL1Lr} implies
		$$
		\|F_\mJ \|^{2-}_{L^1 (S^{n-1}\times\R^n)^{J} L^{2+}(\R^{J})}\les  \int_{\R^{nJ}}     
		\big[\prod_{j\not \in\mJ}^J |k_j|^{0-}   \big]   |  K_J(k_1, \dots,  k_J )|^{2-} dk_1\dots dk_J.
		$$
		By Hardy's inequality, the integral in $k_J$ is
		\begin{multline*}  
		 \les \int \big||D_{k_J}|^{0+} \mF(V(\cdot)e^{ik_{J-1}\cdot})(k_J)\big|^{2-}  dk_J
		 \les \int \big|  \mF(\la \cdot\ra^{0+} V(\cdot)e^{ik_{J-1}\cdot})(k_J)\big|^{2-}dk_J \\ \les \int \big|  \mF(\la \cdot\ra^{0+} V(\cdot))(k_J)\big|^{2-}dk_J
\les \Big[ \int  \la k_J\ra^{0+}\big|  \mF(\la \cdot\ra^{0+}  V(\cdot))(k_J)\big|^{2}dk_J\Big]^{\frac{2-}2}		\les   \|  \la \cdot \ra^{0+} V(\cdot ) \|_{H^{0+}}^{2-}.  
		\end{multline*}   
	Repeating the same argument in the remaining variables yield 
	$$
		\|F_\mJ \|_{L^1 (S^{n-1}\times\R^n)^{J} L^{2+}(\R^{J})}\les \|  \la \cdot \ra^{0+} V(\cdot ) \|_{H^{0+}}^J.
		$$	
		Similar modifications in the other inequalities imply the claim in this case. 
		
		When $n>4\alpha-1$, we apply the inequalities with $0<\delta \ll 1$ and $q=\frac{n-1-\delta}{{n-2\alpha}} $, $r=\frac{n-1-\delta}{{2\alpha-1}-\delta}$ to obtain
		$$\|F_\mJ \|_{L^1 (S^{n-1}\times\R^n)^{J} L^r(\R^{J})}\les \Big[\int_{\R^{nJ}}   \prod_{j\not\in \mJ }   |k_j|^{0- }    | K_J(k_1, \dots,  k_J )|^q dk_1\dots dk_J\Big]^{1/q}
		\les \|\mF(\la \cdot \ra^{0+} V(\cdot))\|_{L^q}^J. $$
		Similarly, we obtain 
		$$
		\|t_1^{\alpha_1}\dots t_J^{\alpha_J}F_\mJ\|_{L^1 (S^{n-1}\times\R^n)^{J} L^r(\R^{J})}\les \|\mF(\la \cdot \ra^{2+} V(\cdot))\|_{L^q}^J, 
		$$
		which implies  that 
		$$
		\big\|\prod_{j=1}^J \la t_j\ra F_\mJ \big\|_{L^1 (S^{n-1}\times\R^n)^{J} L^{\frac{n-1-\delta}{{2\alpha-1}-\delta}}(\R^{J})}   \les 
		\| \mF(\la x\ra^{2+} V)\|_{L^{\frac{n-1-\delta}{{n-2\alpha}}}}^J.  
		$$
		Interpolating the two bounds  we obtain (with $\sigma>\frac{{n-2\alpha}}{n-1-\delta}$)
		$$
		\big\|\prod_{j=1}^J \la t_j\ra^\sigma F_\mJ\big \|_{L^1 (S^{n-1}\times\R^n)^{J} L^{\frac{n-1-\delta}{{2\alpha-1}-\delta}}(\R^{J})}   \les 
		\| \mF(\la x\ra^{2\sigma} V)\|_{L^{\frac{n-1-\delta}{{n-2\alpha}}}}^J, 
		$$
		which implies the claim by H\"older's inequality in the $t$ integrals.  
		
	\end{proof}
The statement in Theorem~\ref{thm:Born} follows by keeping track of the relationship between $q,r,\sigma$ and $\delta$ in the proof above.

\section{Low Energies}\label{sec:low}

In this section we  prove the low energy result, that for sufficiently large $\ell$ the tail of the Born series \eqref{eqn:born identity} extends to a bounded operator on $L^p(\R^n)$.    It is convenient to use a change of variables to respresent $W_{low,\ell}$ as  
$$
	\frac{\alpha}{\pi i}\int_{0}^\infty \chi(\lambda) \lambda^{2\alpha -1} (\mR_0^+(\lambda^{2\alpha }) V)^{\ell } \mR_V^+(\lambda^{2\alpha })  (V\mR_0^+(\lambda^{2\alpha }) )^\ell V   [\mR_0^+(\lambda^{2\alpha })-\mR_0^-(\lambda^{2\alpha })]  \, d\lambda
$$
We begin by using the symmetric resolvent identity on the perturbed resolvent $\mR_V^+(\lambda^{2\alpha })$. With $v=|V|^{\f12}$, $U(x)=1$ if $V(x)\geq 0$ and $U(x)=-1$ if $V(x)<0$, we define $M^+(\lambda)=U+v\mR_0^+(\lambda^{2\alpha })v$. Recall that $M^+$ is invertible on $L^2$   in a sufficiently small neighborhood of $\lambda=0$ provided that zero is a regular point of the spectrum. 
Using the symmetric resolvent identity, one has
$$
\mR_V^+(\lambda^{2\alpha })V=\mR_0^+(\lambda^{2\alpha })vM^+(\lambda)^{-1}v.
$$
We select the cut-off $\chi$ to be supported in this neighborhood.  Therefore, we have
$$
	W_{low,\ell}=	\frac{\alpha}{\pi i}\int_{0}^\infty \chi(\lambda) \lambda^{2\alpha -1} \mR_0^+(\lambda^{2\alpha }) v\Gamma_\ell(\lambda) v   [\mR_0^+(\lambda^{2\alpha })-\mR_0^-(\lambda^{2\alpha })]  \, d\lambda,
$$
where $\Gamma_0(\lambda):= M^+(\lambda)^{-1} $ and for $\ell\geq 1$
\be\label{Gammalambda}
	\Gamma_\ell(\lambda):= Uv\mR_0^+(\lambda^{2\alpha }) \big(V\mR_0^+(\lambda^{2\alpha })\big)^{\ell-1} vM^+(\lambda)^{-1}v \big( \mR_0^+(\lambda^{2\alpha })V\big)^{\ell-1} \mR_0^+(\lambda^{2\alpha })vU.
\ee
To state the main result of this section, we define an operator $T:L^2\to L^2$ with integral kernel $T(x,y)$ to be absolutely bounded if the operator with kernel $|T(x,y)|$ is bounded on $L^2$. 
\begin{prop}\label{lem:low tail low d} Fix  $n>2\alpha > 2$  and $0<\eta $. Let $\Gamma$ be a $\lambda $ dependent absolutely bounded operator. Let 
$$
\widetilde \Gamma (x,y):= \sup_{0<\lambda <\lambda_0}\Big[|\Gamma(\lambda)(x,y)|+ \sup_{1\leq k\leq  \lceil \tfrac{n}2\rceil +1  } \big|\lambda^{k-\eta} \partial_\lambda^k \Gamma(\lambda)(x,y)\big| \Big].
$$
For $2\alpha <n<4\alpha $ assume that $\widetilde \Gamma $ is bounded on $L^2$, and for $n\geq 4\alpha $  assume that $\widetilde\Gamma $ satisfies   \be\label{eq:tildegamma}
\widetilde \Gamma (x,y)  \les \la x\ra^{-\frac{n}2-}\la y\ra^{-\frac{n}2-}.
\ee Then the operator with kernel 
	\be\label{Kdef}
	K(x,y)=\int_0^\infty \chi(\lambda) \lambda^{2\alpha -1} \big[\mR_0^+(\lambda^{2\alpha }) v \Gamma(\lambda)  v [\mR_0^+(\lambda^{2\alpha }) -\mR_0^-(\lambda^{2\alpha })]\big](x,y)   d\lambda 
	\ee
	is bounded on $L^p$ for $1\leq p\leq \infty$ provided that $|V(x)|\les \langle x\rangle^{-\beta}$ for some $\beta>n$.
\end{prop}
In contrast to the case when $\alpha\in \mathbb N$ considered in \cite{EGWaveOp,EGWaveOp2} is that $n-2\alpha$ can be close to zero.  In the integer order case, see Proposition~2.1 in \cite{EGWaveOp2}, the derivatives in $\widetilde{\Gamma}$ are multiplied by $\lambda^{k-1}$ as one can gain a full power of $\lambda$ smallness that is not accessible here.  Instead the parameter $\eta>0$ is used to account for this lack of further smallness when $n-2\alpha<1$, though we leave some flexibility here.
Note that boundedness of the contribution of the tail follows from this proposition and the following 
\begin{lemma}\label{lem:Gamma}
 Fix  $n>2\alpha \geq 2$. Assume that $|V(x)|\les \la x\ra^{-\beta}$, where  $\beta>n+4$ when $n$ is odd and $\beta>n+3$ when $n$ is even.  Also assume that zero is a regular point of the spectrum of $H$.
 Then, for some $\eta>0$,  the operator $\Gamma_\ell(\lambda) $ defined in \eqref{Gammalambda} satisfies the hypothesis of Proposition~\ref{lem:low tail low d} for all $\ell$ when $2\alpha <n<4\alpha $ and for all sufficiently large $\ell$ when $n\geq 4\alpha $.
\end{lemma}
We prove Proposition~\ref{lem:low tail low d} below, and provide the argument for Lemma~\ref{lem:Gamma} in Section~\ref{sec:Minv}.  
To prove these results  we need the following representations of  the free resolvent given in \cite{EGG frac disp}, which were inspired by Lemmas~3.2 and 6.2 in \cite{EGWaveOp}.

\begin{prop}\label{prop:F}

		Fix $\alpha>0$ and $n\in \mathbb N$ with $n>2\alpha$. Then,   we have the representations, with $r=|x-y|$, 
		$$
		\mR_0^+(\lambda^{2\alpha})(x,y)
		=  \frac{e^{i\lambda r}}{r^{n-2\alpha}} F(\lambda r ), \text{ and}$$ 
		\be\label{Frep nbig}
		[\mR_0^+(\lambda^{2\alpha})-\mR_0^-(\lambda^{2\alpha})](x,y)= \lambda^{n-2\alpha}  \big[ e^{i\lambda r}F_+(\lambda r)+e^{-i\lambda r}F_-(  \lambda r)\big],
		\ee
		where, for all $0\leq N\leq \frac{n+1+4\alpha}2$, 
		\be\label{Fbounds nbig}
		|\partial_\lambda^N F(\lambda r)|\les \lambda^{-N} \la  \lambda r \ra^{\frac{n+1}2 -2\alpha},\, \,\,\, \,\,\, 
		|\partial_\lambda^N F_\pm(\lambda r)|\les \lambda^{-N} \la  \lambda r\ra^{-\frac{n-1}2}.
		\ee
Further,  
for all $1\leq N\leq \frac{n+1+4\alpha}2$ we have
\be\label{Fbounds2 nbig}
|\partial_\lambda^N F(\lambda r)|\les\lambda^{-N}(\lambda r)^{\min(1,n-2\alpha,2\alpha -)} \la  \lambda r \ra^{\frac{n+1}2 -2\alpha},\quad 
|\partial_\lambda^N F_{\pm}(\lambda r)|\les\lambda^{-N}(\lambda r) \la  \lambda r \ra^{-\frac{n-1}2},
\ee 
which improves the estimate above for $\lambda r\les 1$.

\end{prop}

We refer the reader to \cite{EGG frac disp} for a full, detailed proof.  For convenience, we provide a brief sketch of the argument.  The free resolvent has a dilation scaling law, so all results may be derived from the case where $\lambda = 1$. The resolvent kernel when $\lambda = 1$ is the Fourier transform of a distribution $\frac{1}{|\xi|^{2\alpha} - (1 + i0)}$.  The results follow from applying an appropriate partition in $\xi$ using smooth cut-off functions and carefully analyzing the Fourier transform of the resulting pieces.

The singularity on the sphere $|\xi| = 1$ dominates the behavior for large $r$. It is a constant multiple of the singularity for the classical Schrodinger resolvent $\frac{1}{|\xi|^2 - (1+i0)}$, plus a remainder that is locally analytic.  The bounds on $F^\pm(r)$ and its derivatives for large $r$ are the same as for the classical Schrodinger resolvent kernel modulo scaling.

The $|\xi|^{2\alpha}$ cusp at the origin has a Fourier transform decaying like $|r|^{-n-2\alpha}$. This gives $F(r)$ an additional term for large $r$ which behaves like $e^{-i r} r^{-4\alpha}$. That is why the estimate in Proposition 3.3 for derivatives of $F(r)$ are only valid up to order $\frac{n+1}{2} + 2\alpha$. Similarly for derivatives of $F_\pm(r)$.

The Fourier transform of any compact cutoff of $\frac{1}{|\xi|^{2\alpha} - (1+i0)}$ is a spherically symmetric, analytic function in $R^n$. Its power series near the origin can only contain even positive integer powers of $r$.

The tail of $\frac{1}{|\xi|^{2\alpha} - (1+i0)}$ can be expanded out in powers of $|\xi|^{-2\alpha}$. For small $r$, the power series expansion of $F(r)$ therefore includes terms of order $r^{2k\alpha}$, with logarithmic corrections if $2k\alpha - n$ is an even integer.  These terms all cancel in the difference of resolvent continuations because the difference is just a surface measure supported on the sphere $|\xi| = 1$.

The power series for $F(r)$ and $F_\pm(r)$ for small $r$ all begin with an $r^0$ term.  Its $N^{th}$ derivative has size zero. The improvement for $N^{th}$ derivative estimates for small $r$ comes from considering the first non-constant term in the power series expansion instead.

\begin{prop}\label{prop:LAP}
		
		Fix $\alpha>\f12$ and $n>2\alpha$.  Assume that $H$ has no embedded eigenvalues.  Then when $\lambda\gtrsim 1$, we have
		$$
			\|\la x\ra^{-\f12-}\mR_V(\lambda^{2\alpha}) \la y\ra^{-\f12 -} \|_{L^2\to L^2} \les  \lambda ^{1-2\alpha},
		$$
		provided that $|V(x)|\les \la x\ra^{-\beta}$ for some $\beta>1$.  Further, for any $j\in \N$ we have
		$$
			\|\la x\ra^{-j-\f12-}\partial_\lambda^j\mR_V(\lambda^{2\alpha}) \la y\ra^{-j-\f12 -} \|_{L^2\to L^2} \les  \lambda ^{1-2\alpha},
		$$
		provided that $|V(x)|\les \la x\ra^{-\beta}$ for some $\beta>1+2j$. 
	 
	\end{prop}

We have the corollaries of Proposition~\ref{prop:F}: 
\begin{corollary}\label{cor:Rj}
	
	Under the conditions of Proposition~\ref{prop:F} we have
	$$
	\sup_{0<\lambda<1}|\mR_0(\lambda^{2\alpha})(x,y)| \les
	|x-y|^{2\alpha -n}+|x-y|^{-\frac{n-1}{2}}.
	$$
	Further, for sufficiently small $\eta>0$, we have
	$$
	\sup_{0<\lambda<1}|\lambda^{\max(0,N-\eta)} \partial_\lambda^N \mR_0(\lambda^{2\alpha })(x,y)| \les 
	|x-y|^{2\alpha +\eta-n}+|x-y|^{N-\frac{n-1}{2}},\,\,\,1\leq N\leq \frac{n+1+4\alpha}2.
	$$
		
\end{corollary}

\begin{proof}
	Using the representation in Proposition~\ref{prop:F},
	when $N=0$ the claim follows from \eqref{Fbounds nbig} and the restriction to $0<\lambda <1$.  For the derivatives, the claim follows by the product rule and considering cases of $\lambda r\ll 1$ and $\lambda r\gtrsim 1$.  When $\lambda r\ll1$, the dominant terms arise when the derivatives act on $F(\lambda r)$ and we use \eqref{Fbounds2 nbig} with $\eta=\min(1,n-2\alpha, 2\alpha-)>0$.  Whenever a derivative acts on the exponential, we use that $r\les \lambda^{-1}$.  
	
	When $\lambda r\gtrsim 1$, the dominant term occurs when the derivatives act on the exponential leading to the factor of $|x-y|^N$.  For the remaining terms
	the claim follows by applying \eqref{Fbounds nbig} and using that $\lambda^{-N}\les r^N$.
\end{proof}

\begin{corollary}\label{cor:E} Let $E(\lambda ) 
(r):=\mR_0^+(\lambda^{2\alpha })(r)-\mR_0^+(0)(r)$. Then, for $\lambda r\ll 1$ and all sufficiently small  $\eta>0$,  we have 
$$
| \partial_\lambda^N E(\lambda)(r)|\les \lambda^{\eta-N} r^{2\alpha -n+\eta}, \,\,\,\,0\leq N\leq\frac{n+1+4\alpha}2.
$$
When $\lambda r\gtrsim 1$, we have 
$$| E(\lambda)(r)| \les r^{\frac{1-n}{2} }\lambda^{\frac{n+1}2-2\alpha } +r^{2\alpha -n}, \text{ and}$$
$$|\partial_\lambda^{N}E(\lambda)(r)| \les r^{\frac{1-n}{2}+N}\lambda ^{\frac{n+1}2-2\alpha }, \,\,\,\,1\leq N\leq\frac{n+1+4\alpha}2.$$  
\end{corollary} 

\begin{proof}
	When $\lambda r\ll 1$ the claim for $N>0$ follows from the product rule and \eqref{Fbounds2 nbig} as in Corollary~\ref{cor:Rj}.  When $N=0$, we use \eqref{Fbounds2 nbig} and the Mean Value Theorem.
	
	When $\lambda r\gtrsim 1$, the dominant terms arise when the derivative acts on the exponential and the bounds for $N>0$ follow from \eqref{Fbounds nbig}.  When $N=0$ the claim follows since both $\mR_0(\lambda^{2\alpha})(r)$ and $\mR_0(0)(r)$ satisfy the claimed bound.
\end{proof}

With these resolvent bounds, we now show that the low energy portion extends to a bounded operator on the full range $1\leq p\leq \infty$.
We say an operator $K$ with integral kernel $K(x,y)$ is admissible if
$$
\sup_{x\in \R^n} \int_{\R^n} |K(x,y)|\, dy+	\sup_{y\in \R^n} \int_{\R^n} |K(x,y)|\, dx<\infty.
$$
By the Schur test, it follows that an operator with admissible kernel is bounded on $L^p(\R^n)$ for all $1\leq p\leq \infty$. We are now ready to prove Proposition~\ref{lem:low tail low d}.

\begin{proof}[Proof of Proposition~\ref{lem:low tail low d}]
	Using the representations in Proposition~\ref{prop:F}  with $r_1=|x-z_1|$ and $r_2:=|z_2-y|$ we see that $K(x,y)$ is the difference of  
 	\be \label{Kdefini}
 	K_\pm(x,y)=\int_{\R^{2n}}  \frac{ v(z_1) v(z_2) }{r_1^{n-2\alpha }   } \int_0^\infty e^{i\lambda (r_1 \pm r_2 )} \chi(\lambda) \lambda^{n-1} \Gamma(\lambda)(z_1,z_2) F(\lambda r_1)F_\pm(\lambda r_2)  d\lambda dz_1 dz_2 .
 	\ee
 	We write
 	$$
 	K(x,y) =: \sum_{j=1}^4 K_{j}(x,y),
 	$$
 	where the integrand in  
 	$K_1$ is  restricted  to the set $r_1,r_2\les 1$, in $K_2$ to the set $r_1 \approx r_2\gg 1 $,
 	in $K_3$ to the set $r_2\gg \la r_1\ra $,  in $K_4$ to the set $r_1 \gg \la r_2\ra$. We define $K_{j,\pm}$ analogously. 
 	
 	Using the bounds of Proposition~\ref{prop:F}  for $\lambda r\ll1$,  we bound the contribution of 	$|K_{1,\pm}(x,y)|$ by 
 	$$
 	\int_{\R^{2n}}  \frac{ v(z_1) v(z_2)\chi_{r_1,  r_2\les 1} }{r_1^{n-2\alpha }   }   \widetilde\Gamma (z_1,z_2)   dz_1 dz_2.
 	$$
 	Therefore
 	$$
 	\int  |K_{1,\pm}(x,y)|dx    \les \big\|| \cdot|^{2\alpha -n}\big\|_{L^1(B(0,1))} \|v\|_{L^2}^2 \|\widetilde \Gamma\|_{L^2\to L^2}\les 1,
 	$$
 	uniformly in $y$. Similarly, provided that $2\alpha <n<4\alpha $,
 	$$\int  |K_{1,\pm}(x,y)|dy\les \|\widetilde\Gamma\|_{L^2\to L^2} \|v\|_{L^2} \|v(\cdot) |x-\cdot|^{2\alpha -n}\|_{L^2}\les 1 
 	$$
 	holds uniformly in $x$. When $n\geq 4\alpha $, we use the decay bound \eqref{eq:tildegamma} on $\widetilde \Gamma$ to obtain
 	$$\int  |K_{1,\pm}(x,y)|dy\les \int \la z_1\ra^{-n-}\la z_2 \ra^{-n-} r_1^{2\alpha -n} dz_1dz_2 \les 1,  
 	$$
 	which implies that $K_1$ is admissible. 
 	
 	For $K_2$, we restrict ourself to $K_{2,-}$ since the $+$ sign is easier to handle. We integrate by parts twice in the $\lambda$ integral when $\lambda |r_1-r_2|\gtrsim1$ (using  Proposition~\ref{prop:F} and the definition of $\widetilde \Gamma$) and estimate directly when $\lambda |r_1-r_2|\ll1$ to obtain 
 	\begin{align} 
 		\nonumber
 		|K_{2,-}(x,y)|\les \int_{\R^{2n}} \frac{v(z_1)  \widetilde\Gamma (z_1,z_2) v(z_2)\chi_{r_1\approx r_2\gg 1 }}{  r_1^{n-2\alpha }  } \int_0^\infty  \chi(\lambda) \lambda^{n-1} \chi(\lambda|r_1-r_2|)  \la\lambda r_1\ra^{1-2\alpha }    d\lambda dz_1 dz_2 \\ \nonumber
 		+ \int_{\R^{2n}} \frac{v(z_1)  \widetilde\Gamma (z_1,z_2) v(z_2)\chi_{r_1\approx r_2\gg 1 }}{  r_1^{n-2\alpha }  } \int_0^\infty  \frac {\chi(\lambda) \lambda^{n-3} \widetilde\chi(\lambda|r_1-r_2|)   \la\lambda r_1\ra^{1-2\alpha }}{|r_1-r_2|^2}     d\lambda dz_1 dz_2\\
 		\les  \int_{\R^{2n}} \frac{v(z_1)  \widetilde\Gamma (z_1,z_2) v(z_2)\chi_{r_1\approx r_2\gg 1 }}{  r_1^{n-2\alpha }  } \int_0^\infty  \frac {\chi(\lambda) \lambda^{n-1}   \la\lambda r_1\ra^{1-2\alpha }}{\la \lambda (r_1-r_2)\ra^2}     d\lambda dz_1 dz_2. \nonumber
 	\end{align}
 	Therefore, passing to the polar coordinates in $x$ integral (centered at $z_1$) and noting 
 	$1-2\alpha <0$, we have 
 	$$
 	\int|K_{2,-}(x,y)| dx 
 	\les  \int_{\R^{2n}} \int_0^1 \int_{r_1\approx r_2\gg 1}  v(z_1)  \widetilde\Gamma (z_1,z_2) v(z_2)    \frac {  \lambda^{n-2\alpha }   }{\la \lambda (r_1-r_2)\ra^2}      d r_1 d\lambda dz_1 dz_2
 	$$
 	$$  \les  \int_{\R^{2n}} \int_0^1 \int_{\R}  v(z_1)  \widetilde\Gamma (z_1,z_2) v(z_2)    \frac {  \lambda^{n-2\alpha -1}   }{\la  \eta \ra^2}      d \eta  d\lambda dz_1 dz_2\les 1, 
 	$$
 	uniformly in $y$. In the second line we defined $\eta=\lambda (r_1-r_2)$ in the $r_1$ integral and used   $n-2\alpha -1>-1$.
 	Since $r_1\approx r_2$, the integral in $y$ can be bounded uniformly in $x$ and hence the contribution of $K_2$ is admissible.  We now consider the contribution of
 	\begin{multline} \label{K4ndef}
 		K_{4,\pm}(x,y)= 
 		\int_{\R^{2n}}  \frac{v(z_1) v(z_2)\chi_{r_1 \gg \la r_2\ra}}{r_1^{n-2\alpha }}  \\  \int_0^\infty  e^{i  \lambda (r_1\pm r_2) } F(  \lambda r_1) \chi(\lambda)   \Gamma(\lambda)(z_1,z_2) \lambda^{n-1}  F_\pm(\lambda r_2)  \ d\lambda dz_1 dz_2 .
 	\end{multline}
 	When $\lambda r_1\les 1$, using \eqref{Fbounds nbig}, we bound $|F_\pm(\lambda r_2)|, |F(\lambda r_1)|\les 1$ and estimate the $\lambda$ integral  by 
 	$r_1^{-n} \widetilde \Gamma(z_1,z_2)$, whose contribution to  $K_4$ is bounded by 
 	$$
 	\int_{\R^{2n}}  \frac{v(z_1) v(z_2) \widetilde \Gamma(z_1,z_2) \chi_{r_1 \gg \la r_2\ra}}{r_1^{2n-2\alpha}} dz_1 dz_2 .
 	$$
 	This is admissible since $n>2\alpha$. 
 	
 	When $\lambda r_1 \gtrsim 1$, we integrate by parts $N=\lceil n/2\rceil +1 $ times 
 	(using \eqref{Fbounds nbig}) to obtain the bound
 	$$
 	\frac1{|r_1\pm r_2|^{N}}\int_0^\infty \Big|\partial_\lambda^{N} 
 	\big[F(  \lambda r_1) \widetilde\chi(\lambda r_1)  \chi(\lambda)  \lambda^{n-1} \Gamma(\lambda)(z_1,z_2)   F_\pm(\lambda r_2) \big] \Big| d\lambda 
 	$$  
 	$$
 	\les r_1^{-N} \sum_{0\leq j_1+j_2+j_3+j_4\leq N, \, j_i\geq 0  }  
 	\int_{\frac1{r_1}}^1    \lambda^{\frac{n+1}2-2\alpha -j_1}  r_1^{\frac{n+1}2-2\alpha }   \lambda^{n-1-j_2}  \big|\partial_\lambda^{j_3}\Gamma(\lambda)(z_1,z_2)  \big|\frac{\lambda^{-j_4}}{\la \lambda r_2\ra^{\frac{n-1}2}  } d\lambda 
 	$$
 	$$
 	\les r_1^{ \frac{n+1}2-2\alpha -N }  \widetilde \Gamma(z_1,z_2) \sum_{0\leq j_1+j_2+j_3+j_4\leq N, \, j_i\geq 0  }  
 	\int_{\frac1{r_1}}^1    \lambda^{\frac{3n-1}2-2\alpha -j_1-j_2-j_3-j_4}        d\lambda 
 	$$
 	$$
 	\les r_1^{ -2\alpha-\frac12 -\{\frac{n}2\} }  \widetilde \Gamma(z_1,z_2) \int_{\frac1{r_1}}^1    \lambda^{n -2\alpha -\frac32 -\{\frac{n}2\}       }  d\lambda    \les  r_1^{ -2\alpha  -\epsilon}\widetilde \Gamma(z_1,z_2) \int_{\frac1{r_1}}^1    \lambda^{n-2\alpha -1-\epsilon}  d\lambda$$
 	$$  \les  r_1^{  -2\alpha - \epsilon}\widetilde \Gamma(z_1,z_2),
 	$$
 	provided that $0<\epsilon<n-2\alpha$. The contribution of this to \eqref{K4ndef} is 
 	$$
 	\les \int_{\R^{2n}}  \frac{v(z_1) v(z_2)\chi_{r_1 \gg \la r_2\ra}}{r_1^{n+\epsilon}} dz_1 dz_2, 
 	$$
which is   admissible as $\epsilon>0$.

 	We now consider $K_3$:
 	\begin{multline} \label{K3ndef1}
 		K_3(x,y)=  
 		\int_{\R^{2n}}  \chi_{r_2 \gg \la r_1\ra} v(z_1) v(z_2)\\ 	\int_0^\infty \lambda^{2\alpha-1} \chi(\lambda)   \mR_0^+(\lambda^{2\alpha }) (r_1)    \Gamma(\lambda)(z_1,z_2) [\mR_0^+(\lambda^{2\alpha})- \mR_0^-(\lambda^{2\alpha})](r_2)  \ d\lambda dz_1 dz_2 .
 	\end{multline}
 	We write
 	$$
 	\mR_0^+(\lambda^{2\alpha })=\mR_0^+(0) + [\mR_0^+(\lambda^{2\alpha })-\mR_0^+(0)]=: G_0+E(\lambda),
 	$$
 	$$
 	\Gamma(\lambda)= \Gamma(0)+[\Gamma(\lambda)-\Gamma(0)]=:\Gamma(0)+\Gamma_1(\lambda).
 	$$
 	Here $G_0=\mR_0^+(0)=c_{n,\alpha}r_1^{2\alpha -n}$.
 	We first consider the contribution of $G_0\Gamma(0)$ to $K_3$:
 	$$
 	\int_{\R^{2n}}  \chi_{r_2 \gg \la r_1\ra} v(z_1) v(z_2) 	G_0(r_1)\Gamma(0)(z_1,z_2) \int_0^\infty \lambda^{2\alpha-1} \chi(\lambda)    [\mR_0^+(\lambda^{2\alpha})- \mR_0^-(\lambda^{2\alpha})](r_2)  \ d\lambda dz_1 dz_2. 
 	$$
 	Identifying the $\lambda$ integral as a constant multiple of the kernel of $\chi((-\Delta)^{\frac1{2\alpha}})$,   we may bound it as $O(\la r_2\ra^{-N})$ for all $N$   since $\chi(|\xi|^{\frac1{\alpha}})$ is Schwartz.  Therefore, we have the bound
 	$$
 	\int_{\R^{2n}}  \chi_{r_2 \gg \la r_1\ra} v(z_1) v(z_2) 	r_1^{2\alpha -n} r_2^{-n-1} \widetilde \Gamma (z_1,z_2) dz_1 dz_2,
 	$$
 	which is admissible by Lemma~\ref{lem:admissible}. 
 	
 	It remains to consider the contributions of $\mR^+_0(\lambda^{2\alpha })\Gamma_1(\lambda) $ and of $ E(\lambda) \Gamma(0)$.
 	The former can be written as 
 	$$
 	\int_{\R^{2n}}  \frac{v(z_1) v(z_2)\chi_{r_2 \gg \la r_1\ra}}{r_1^{n-2\alpha }}  \int_0^\infty  e^{i  \lambda (r_1\pm r_2) } F(  \lambda r_1) \chi(\lambda)  \lambda^{n-1}   \Gamma_1(\lambda)(z_1,z_2)  F_\pm(\lambda r_2)  \ d\lambda dz_1 dz_2 .
 	$$
 	When $\lambda r_2\ll1$, using $|\Gamma_1(\lambda)|\les \lambda^\eta \widetilde \Gamma$, which follows from the mean value theorem, and Proposition~\ref{prop:F} to directly integrate in $\lambda$, we obtain the bound
 	$$
 	\int_{\R^{2n}}  \frac{v(z_1) v(z_2)\chi_{r_2 \gg \la r_1\ra}}{r_1^{n-2\alpha }r_2^{n+\eta}}\widetilde \Gamma(z_1,z_2)    dz_1 dz_2,
 	$$
 	which is admissible by Lemma~\ref{lem:admissible} as $\eta>0$. When $\lambda r_2 \gtrsim 1$, integrating by parts $N=\lceil n/2\rceil +1$ times, we have the bound
 	\be\label{K3former}
 	\int_{\R^{2n}}  \frac{v(z_1) v(z_2)\chi_{r_2 \gg \la r_1\ra}}{r_1^{n-2\alpha }|r_1\pm r_2|^N}  \int_0^\infty  \Big| \partial_\lambda^N \big[ F(  \lambda r_1) \chi(\lambda) \widetilde \chi(\lambda r_2)  \lambda^{n-1}   \Gamma_1(\lambda)(z_1,z_2)  F_\pm(\lambda r_2)\big]\Big|  \ d\lambda dz_1 dz_2 .
 	\ee
 	We estimate the $\lambda $ integral by (noting that $\lambda^{j_3} |\partial^{j_3}_\lambda \Gamma_1|\les \lambda^\eta \widetilde \Gamma$ and using Proposition~\ref{prop:F})
 	$$
 	\les r_2^{-\frac{n-1}2} \widetilde\Gamma(z_1,z_2)  \sum_{0\leq j_1+j_2+j_3+j_4\leq N, \, j_i\geq 0  }  
 	\int_{\frac1{r_2}}^1  \la  \lambda r_1\ra^{\frac{n+1}2-2\alpha } \lambda^{-j_1} \lambda^{n-1-j_2}      \lambda^{\eta-j_3}  \lambda^{-\frac{n-1}2-j_4} d\lambda 
 	$$
 	$$
 	\les r_2^{-\frac{n-1}2} \widetilde\Gamma(z_1,z_2)   
 	\int_{\frac1{r_2}}^1  \la  \lambda r_1\ra^{\frac{n+1}2-2\alpha }   \lambda^{\frac{n+1}2-\lceil \frac{n}2\rceil +\eta-2} d\lambda 
 	$$
 	$$
 	\les r_2^{-\frac{n-1}2} \widetilde\Gamma(z_1,z_2) 
 	\Big(\int_{\frac1{r_2}}^{\min(\frac1{r_1},1)}  \lambda^{\eta -\{ \frac{n}2\}-\frac32  }      d\lambda +  \int_{\min(\frac1{r_1},1)}^1 r_1^{\frac{n+1}2-2\alpha }  \lambda^{\frac{n}2-2\alpha+\eta -\{ \frac{n}2\}-1}       d\lambda  \Big).
 	$$
For $0<\eta\ll 1$,  	the first integral is at most $r_2^{-\eta +\{ \frac{n}2\}+\frac12}$. Its contribution to \eqref{K3former} is at most
 	$$
 	\int_{\R^{2n}}  \frac{v(z_1) v(z_2)\chi_{r_2 \gg \la r_1\ra}}{r_1^{n-2\alpha }r_2^{n+\eta}}\widetilde \Gamma(z_1,z_2)    dz_1 dz_2,
 	$$
 	which is admissible by Lemma~\ref{lem:admissible}. Similarly,  the second integral is bounded by  $r_1^{n-2\alpha } r_1^{\frac12+\{\frac{n}2\}}$ after multiplying the integrand by $(\lambda r_1)^{\frac{n-1}2}$ (assuming that $\eta\leq n-2\alpha$). Contribution of this to \eqref{K3former} is at most
 	$$
 	\int_{\R^{2n}}  \frac{v(z_1) v(z_2)\chi_{r_2 \gg \la r_1\ra} \la r_1\ra^{\frac12+\{\frac{n}2\}}}{ r_2^{n+\frac12+\{\frac{n}2\}}}\widetilde \Gamma(z_1,z_2)    dz_1 dz_2,
 	$$
 	which, by Lemma~\ref{lem:admissible}, is also admissible.

 	We now consider the contribution of $E(\lambda)\Gamma(0)$:  
 	\be\label{egamma}
 	\int_{\R^{2n}}   v(z_1) v(z_2)\chi_{r_2 \gg \la r_1\ra}\Gamma(0)(z_1,z_2)  \int_0^\infty  e^{\pm  i  \lambda   r_2 }  E(\lambda)(r_1) \chi(\lambda)   \lambda^{n-1}  F_\pm(\lambda r_2)  \ d\lambda dz_1 dz_2 .
 	\ee
 	Using Proposition~\ref{prop:F} and Corollary~\ref{cor:E} when $\lambda r_1\ll 1$.  Using this 
 	when $\lambda r_2\ll 1$ and using $|\Gamma(0)(z_1,z_2)|\leq \widetilde\Gamma(z_1,z_2)$, we bound \eqref{egamma} by direct estimate by 
 	$$
 	\int_{\R^{2n}}  \frac{ v(z_1) v(z_2)\chi_{r_2 \gg \la r_1\ra}\widetilde\Gamma(z_1,z_2) }{r_1^{n-2\alpha -\eta} r_2^{n+\eta} }  dz_1 dz_2,
 	$$  
 	which is admissible by Lemma~\ref{lem:admissible} since $\eta , n-2\alpha >0$. 
 	
 	When $\lambda r_2\gtrsim 1$ and $\lambda r_1 \ll 1$, we integrate by parts $N=\lceil n/2\rceil +1  $ times to obtain
 	$$
 	\int_{\R^{2n}}   r_2^{ -N}  v(z_1) v(z_2)\chi_{r_2 \gg \la r_1\ra}\widetilde \Gamma (z_1,z_2)   \int_0^\infty  \Big|\partial_\lambda^N\big[  E(\lambda)(r_1) \chi(\lambda)  \chi(\lambda r_1)  \widetilde  \chi(\lambda r_2)  \lambda^{n-1}  F_\pm(\lambda r_2)\big]\Big|  \ d\lambda dz_1 dz_2. 
 	$$
 	Using Corollary~\ref{cor:E} and Proposition~\ref{prop:F}, we bound this by 
 	$$
 	\int_{\R^{2n}}  r_2^{-N+\frac {1-n}2} r_1^{\eta+2\alpha -n} v(z_1) v(z_2)\chi_{r_2 \gg \la r_1\ra}\widetilde \Gamma (z_1,z_2)  \int_{r_2^{-1}}^1 \lambda^{\eta -\{ \frac{n}2\}-\frac32 }   \ d\lambda dz_1 dz_2 
 	$$
 	$$
 	\les  \int_{\R^{2n}}  r_2^{-n-\eta} r_1^{\eta+2\alpha -n} v(z_1) v(z_2)\chi_{r_2 \gg \la r_1\ra}\widetilde \Gamma (z_1,z_2)  dz_1 dz_2,
 	$$
 	which is again admissible by Lemma~\ref{lem:admissible}.

 	It remains to consider the case $\lambda r_1\gtrsim 1$. Integrating by parts once we rewrite the $\lambda$ integral in \eqref{egamma} as  
 	\begin{align}\label{eq:K3 bad1}
 		&\frac1{r_2} \int_0^\infty  e^{\pm  i  \lambda   r_2 }  \partial_\lambda[E(\lambda)(r_1)] \ \widetilde\chi(\lambda r_1) \chi(\lambda)   \lambda^{n-1}  F_\pm(\lambda r_2)  \ d\lambda \\
 		&+ \frac1{r_2} \int_0^\infty  e^{\pm  i  \lambda   r_2 }  E(\lambda)(r_1)\ \partial_\lambda \big[ \widetilde\chi(\lambda r_1) \chi(\lambda)   \lambda^{n-1}  F_\pm(\lambda r_2)  \big] \ d\lambda. \label{eq:K3 bad2}
 	\end{align}
 	For the second integral, \eqref{eq:K3 bad2},  we integrate by parts $N=\lceil \frac{n}2\rceil$ more times using  Proposition~\ref{prop:F}, to obtain the bound
 	$$
 	\frac1{r_2^{N+\frac{n}2+\frac12}} \sum_{j_1+j_2 \leq N,  \ 0\leq j_1, j_2} \int_{r_1^{-1}}^1  \big| \partial_\lambda^{j_1} [E(\lambda)(r_1)]\big|  \      \lambda^{\frac{n-3}2-j_2 }   \ d\lambda.  
 	$$
 	Using Corollary~\ref{cor:E} we bound this by
 	$$
 	\les  \frac1{r_2^{N+\frac{n}2+\frac12 }}\Big[\int_{r_1^{-1}}^1  r_1^{2\alpha -n}      \lambda^{\frac{n-3}2-N }   \ d\lambda +  \sum_{j_1+j_2 \leq N,  \ 0\leq j_1, j_2} \int_{r_1^{-1}}^1 r_1^{j_1+\frac{1-n}2}    \lambda^{n-2\alpha  -1- j_2 }   \ d\lambda\Big].
 	$$
 	The first integral takes care of the  additional term that arises in Corollary~\ref{cor:E} (for $\lambda r\gtrsim 1$)  in the case $j_1=0$.    Letting $\{n/2\}=  n/2-\lfloor n/2\rfloor$, we bound this by
 	
 	$$
 	\les \frac{r_1^{\{n/2\}+\frac12 +2\alpha -n}+r_1^{\{n/2\}+\frac12}}{r_2^{n+\{n/2\}+\frac12}}
 	\les \frac{r_1^{\{n/2\}+\frac12}}{r_2^{n+\{n/2\}+\frac12}},
 	$$
 	whose contribution is admissible   by Lemma~\ref{lem:admissible2} since $r_2\gg \la r_1\ra$. 
 	
 	For the first integral, \eqref{eq:K3 bad1}, we integrate by parts $N=\lceil \frac{n}2\rceil$ more times after pulling out the phase $e^{i\lambda r_1}$ to obtain the bound
 	$$\frac1{r_2^{ \frac{n}2+\frac12} |r_1\pm r_2|^N} \sum_{j_1+j_2 \leq N,  \ 0\leq j_1, j_2} \int_{r_1^{-1}}^1   \big| \partial_\lambda^{j_1}[\widetilde E(\lambda)(r_1)] \big| \   \lambda^{\frac{n-1}2-j_2 } \ d\lambda 
 	$$
 	$$
 	\widetilde E(\lambda)(r_1):= e^{- i\lambda r_1} \partial_\lambda[E(\lambda)(r_1)]
 	$$
 	Using the representation in Proposition~\ref{prop:F}, we bound this by
 	\begin{multline*}
 		\frac1{r_2^{ n+\{n/2\}+\frac12}  } \sum_{j_1+j_2 \leq N,  \ 0\leq j_1, j_2} \int_{r_1^{-1}}^1  r_1 \frac{(\lambda r_1)^{\frac{n+1}2-2\alpha }}{r_1^{n-2\alpha }}    \lambda^{\frac{n-1}2-j_1-j_2 } \ d\lambda \\
 		\les \frac1{r_2^{ n+\{n/2\}+\frac12}  r_1^{\frac{n-3}2} }  \int_{r_1^{-1}}^1 \lambda^{\frac{n }2-2\alpha  -\{n/2\} } \ d\lambda
 		\les  \frac{r_1^{\{n/2\}+\frac12}}{r_2^{n+\{n/2\}+\frac12}},
 	\end{multline*}
 	which is admissible by Lemma~\ref{lem:admissible2}.
 \end{proof}

\section{Proof of Lemma~\ref{lem:Gamma}}\label{sec:Minv}
For small energies, it remains only to prove Lemma~\ref{lem:Gamma} stating  that the operators $\Gamma_\ell(\lambda)$ defined in \eqref{Gammalambda} satisfy the bounds needed to apply Proposition~\ref{lem:low tail low d}. This follows, with some modifications, from the discussion preceeding Lemma~3.5 in \cite{EGWaveOp}, also see Section 4 of \cite{EGWaveOp2}.  We briefly sketch the argument here.  One of the differences here is that $n-2\alpha$ can be close to zero, in which case we cannot gain a full power of $\lambda$ in the bound \eqref{R_k} below.  As discussed after the statement of Proposition~\ref{lem:low tail low d}, we utilize a parameter $\eta>0$ to account for this positive difference.

We write $n_{\star}$ to denote $n+4$ if $n$ is odd and $n+3$ if $n$ is even. By Corollary~\ref{cor:Rj}, for a sufficiently small $\eta>0$, the operator $R_j$ with kernel
\be\label{R_k}
R_j(x,y):=v(x)v(y)\sup_{0<\lambda<1}|\lambda^{\max(0,j-\eta)} \partial_\lambda^j \mR_0^+(\lambda^{2\alpha })(x,y)|
\ee
is bounded on $L^2(\R^n)$ for $0\leq j\leq \lceil \frac{n}2\rceil+1$ provided that $|V(x)|\les \la x \ra^{-\beta}$ for some $\beta>n_{\star}$.   This follows the argument from Section 4 of \cite{EGWaveOp2}.

Similarly by Corollary~\ref{cor:E}, $\mathcal E(\lambda):=v[\mR^+_0(\lambda^{2\alpha })-\mR^+_0(0)]v$ satisfies 
$$
\|\mathcal E(\lambda)\|_{L^2\to L^2} \les \lambda^\eta.
$$

	Now, we define the operator
	$$
	T_0:=U+v\mR_0^+(0)v=M^+(0).
	$$
	By the assumption that zero energy is regular, $T_0$ is invertible with absolutely bounded inverse, see \cite{EGG frac disp}. Note that by a Neumann series expansion and the invertibility of $T_0$ we have 
	$$
	[M^+(\lambda)]^{-1} =\sum_{k=0}^\infty (-1)^k T_0^{-1} (\mE(\lambda)T_0^{-1})^{k}.
	$$
	The series converges in the operator norm on $L^2$ for sufficiently small $\lambda$. 
 By the resolvent identity the operator
	$ \partial_\lambda^N[M^+(\lambda)]^{-1} $ is a linear combination of operators of the form 
	$$
	[M^+(\lambda)]^{-1} \prod_{j=1}^J\big[v\big( \partial_\lambda^{N_j}\mR_0^{+}(\lambda^{2\alpha })\big)v[M^+(\lambda)]^{-1}\big],
	$$
	where $\sum N_j=N$ and each $N_j\geq 1$. From the discussion above on $R_j$'s  this representation implies  that  
	\be\label{eq:MGamma}
 \sup_{0<\lambda<\lambda_0}\lambda^{\max(0,N-\eta)} | \partial_\lambda^N[M^+(\lambda)]^{-1}(x,y)|
	\ee
	is bounded in $L^2$ for $N=0,1,\ldots,\lceil \frac{n}{2}\rceil+1$ provided that $\beta >n_\star$.   

Recalling the definition of $\Gamma_\ell(\lambda)$, \eqref{Gammalambda}, and noting the $L^2$ boundedness of $R_j$'s above we see that
$$ \sup_{0<\lambda<\lambda_0}\lambda^{\max(0,N-\eta)} \big| \partial_\lambda^N \big(Uv\mR_0^+(\lambda^{2\alpha }) \big(V\mR_0^+(\lambda^{2\alpha })\big)^{\ell-1} v \big)(x,y)\big|
$$
is bounded on $L^2$. This yields   Lemma~\ref{lem:Gamma}	when $2\alpha <n<4\alpha $. 

For  $n\geq4\alpha $, Lemma~\ref{lem:Gamma} follows by writing
$$
\Gamma_k(\lambda)=Uv A(\lambda)vM^{-1}(\lambda)v A(\lambda)vU,
$$
where
\begin{align}\label{eq:Alambda}
	A(\lambda, z_1,z_2) =  \big[ \big(\mR_0^+(\lambda^{2\alpha })  V\big)^{\ell-1}\mR_0^+(\lambda^{2\alpha })\big](z_1,z_2).
\end{align}  

If   $ \ell-1 $ is sufficiently large depending on $n,\alpha$ and $|V(x)|\les \la x \ra^{-\frac{n_\star}2-}$, then 
\begin{align*}
	\sup_{0<\lambda <1}| \lambda^{\max\{0,k-\eta\}} \partial_\lambda^k 	A(\lambda, z_1,z_2)|&\les \la z_1 \ra^{\frac32+\{\frac{n}2\}}  \la z_2\ra^{\frac32+\{\frac{n}2\}},
\end{align*}
for $0\leq \ell\leq \lceil \frac{n}2\rceil+1$. This follows from the pointwise bounds on $R_j$ above. The iteration of the resolvents smooths out the local singularity $|x-\cdot|^{2\alpha -n}$.  Each iteration improves the local singularity by $2\alpha $, so that after $j$ iterations the local singularity is of size $|x-\cdot|^{2\alpha j-n}$.  Selecting $\ell-1$ large enough ensures that the local singularity is completely integrated away.  This yields Lemma~\ref{lem:Gamma}, see \cite{EGWaveOp,EGG frac disp} for more details.

We recall Lemmas~4.1 and 4.2 from \cite{EGWaveOp2}, which we used in the proof of Proposition~\ref{lem:low tail low d}, which we state here for the convenience of the reader.
\begin{lemma}\label{lem:admissible}
	Let $K$ be an operator with integral kernel $K(x,y)$ that satisfies the bound
	$$
		|K(x,y)|\les \int_{\R^{2n}} \frac{v(z_1)v(z_2) \widetilde \Gamma(z_1,z_2) \chi_{\{|y-z_2|\gg \la z_1-x\ra \} }}{|x-z_1|^{n-2\alpha -k} |z_2-y|^{n+\ell} } \, dz_1 \, dz_2
	$$
	for some $0\leq k\leq n-2\alpha $ and $\ell>0$.  Then, under the hypotheses of Lemma~\ref{lem:low tail low d}, the kernel of $K$ is admissible, and consequently $K$ is a bounded operator on $L^p(\R^n)$ for all $1\leq p\leq \infty$.
	
\end{lemma}

\begin{lemma}\label{lem:admissible2}
	
	Let $K$ be an operator with integral kernel $K(x,y)$ that satisfies the bound
	$$
	|K(x,y)|\les \int_{\R^{2n}} \frac{v(z_1)v(z_2) \widetilde \Gamma(z_1,z_2) \chi_{\{|y-z_2|\gg \la z_1-x\ra \}} |x-z_1|^{\ell} }{  |z_2-y|^{n+\ell} } \, dz_1 \, dz_2
	$$
	for some $\ell>0$.  Then, under the hypotheses of Lemma~\ref{lem:low tail low d}, the kernel of $K$ is admissible, and consequently $K$ is a bounded operator on $L^p(\R^n)$ for all $1\leq p\leq \infty$.
\end{lemma}

\section{High Energy}\label{sec:high}

Since we can control the contribution of the Born series to arbitrary length, and the low energy portion of the tail of the series in \eqref{eqn:born identity}, we need only consider the tail when $\lambda\gtrsim 1$ and show that
$$
	\int_0^\infty \widetilde \chi(\lambda)\lambda^{2\alpha-1}[(\mathcal R_0^+ V)^\ell V\mathcal R_V^+ (V\mathcal R_0^+)^\ell V\mathcal R_0^\pm ](\lambda^{2\alpha})\, d\lambda
$$
extends  to  bounded operators on $L^p(\R^n)$ provided $\ell$ is sufficiently large.  In all cases we assume there are no positive eigenvalues of $H$.  The argument   is identical to that in \cite{EGWaveOp} for the integer order Schr\"odinger case using the bounds on the resolvents in Proposition~\ref{prop:F} and the limiting absorption principle in Proposition~\ref{prop:LAP}, which are tailored to the fractional Schr\"odinger operators.  This allows us to pointwise dominate the integral kernel of the tail of the Born series with an admissible kernel:

\begin{prop}\label{prop:high odd}
	
	We have the bound 
	\begin{multline}\label{eqn:bs tail int}
		\bigg|\int_0^\infty \widetilde \chi(\lambda)\lambda^{{2\alpha-1}}
		(\mathcal R_0^+(\lambda^{2\alpha }) V)^{\ell +1}  \mathcal R_V^+(\lambda^{2\alpha })V (\mathcal R_0^+(\lambda^{2\alpha })V)^{\ell } \mathcal R_0^\pm (\lambda^{2\alpha })(x,y)\, d\lambda\bigg|\\
		\les \frac{1}{\la |x|-|y|\ra^{\frac{n+3}{2}} \la x\ra^{\frac{n-1}{2} } \la y \ra^{\frac{n-1}{2}}},
	\end{multline}
	provided $\ell $ is sufficiently large, and $|V(x)|\les \la x\ra^{-\beta}$ for some $\beta>n_\star$.
	
\end{prop}

The proof is a straightforward modification of the proof of Propositions~5.3 and 6.5 in \cite{EGWaveOp}.  
To complete the claim, we use Lemma~5.2 from \cite{EGWaveOp} (also see Lemma~3.1 in \cite{YajNew} and Lemma~2.1 in \cite{GG4wave}):
\begin{lemma}\label{lem:ptwise kernel}
	Suppose that $K$ is an integral operator whose kernel obeys the pointwise bounds
	\begin{align} \label{eqn:kernels}
	|K(x,y)|\les \frac{1}{\la x\ra^{\frac{n-1}{2}} \la y \ra^{\frac{n-1}{2}} \la |x|-|y|\ra^{\frac{n+1}{2}+\epsilon}}.
	\end{align}
	Then $K$ is admissible provided that $\epsilon>0$.
	
\end{lemma}
The proposition and lemma imply that the kernel is admissible and hence the tail extends to a bounded operator on $L^p(\R ^n)$ for all $1\leq p\leq \infty$.

\section*{Data Availability}
 No data was used for the research described in the article.

\end{document}